\documentclass[a4paper,11pt,reqno]{amsart}

\usepackage{graphicx}
\usepackage{amsmath,amsopn,amssymb,amsfonts,stmaryrd}
\usepackage{verbatim}
\usepackage{amsthm}
\usepackage{mathtools}
\usepackage{color}
\usepackage{enumitem}
\usepackage[framemethod=TikZ]{mdframed}
\usepackage{bbm}
\usepackage{mathrsfs}
\usepackage{booktabs}
\usepackage{caption}
\usepackage{bm}
\usepackage[top=1.2in, bottom=1.2in, left=0.9in, right=0.9in]{geometry}
\usepackage{makecell}
\usepackage{latexsym}

\usepackage{cleveref}

\theoremstyle{plain}
\newtheorem{theorem}{Theorem}[section]

\newtheorem{conjecture}[theorem]{Conjecture}
\newtheorem{lemma}[theorem]{Lemma}
\newtheorem{proposition}[theorem]{Proposition}
\newtheorem*{remark}{Remark}

\theoremstyle{definition}

\newtheorem*{definition*}{Definition}

\newtheorem*{example*}{Example}
\newtheorem{question}{Question}

\numberwithin{theorem}{section}
\numberwithin{equation}{section}

\usepackage{booktabs} 
\usepackage{array} 
\usepackage{verbatim} 
\usepackage{enumitem}
\usepackage{subfig} 

\newcommand{\Z}{\mathbb{Z}}

\newcommand{\N}{\mathbb{N}}
\newcommand{\R}{\mathbb{R}}
\newcommand{\C}{\mathbb{C}}

\newcommand{\Nc}{\mathcal{N}}

\newcommand{\Cc}{\mathcal{C}}

\newcommand{\Ec}{\mathcal{E}}
\newcommand{\Ic}{\mathcal{I}}
\renewcommand{\Mc}{\mathcal{M}}

\newcommand{\Dc}{\mathcal{D}}

\renewcommand{\b}{\beta}
\renewcommand{\d}{\delta}

\renewcommand{\k}{\kappa}

\newcommand{\g}{\gamma}
\newcommand{\om}{\omega}
\newcommand{\vr}{\varrho}

\newcommand{\z}{\zeta}
\newcommand{\La}{\Lambda}
\newcommand{\De}{\Delta}
\newcommand{\Ga}{\Gamma}

\newcommand{\cb}[1]{\left\{{#1}\right\}}

\newcommand{\crank}{\operatorname{crank}}

\newcommand{\del}{\partial}

\newcommand{\im}{\operatorname{Im}}

\renewcommand{\pmod}[1]{\,\,\left({\rm mod}\,\,{#1}\right)}

\newcommand{\pd}[2]{\frac{\partial{#1}}{\partial{#2}}}

\newcommand{\flo}[1]{\lfloor #1\rfloor}
\newcommand{\pflo}[1]{\left\lfloor #1\right\rfloor}

\newcommand{\pcei}[1]{\left\lceil #1\right\rceil}

\newcommand{\rank}{\operatorname{rank}}
\newcommand{\rb}[1]{\left({#1}\right)}
\newcommand{\re}{\operatorname{Re}}

\newcommand{\sqb}[1]{\left[{#1}\right]}

\newcommand{\vb}[1]{\left| {#1} \right|}

\newcommand{\mc}{\mathcal}
\newcommand{\mf}{\mathfrak}

\allowdisplaybreaks

\newcommand{\bdd}{\begin{center}\begin{tikzcd}}
\newcommand{\bd}{\begin{tikzcd}}
\newcommand{\edd}{\end{tikzcd}\end{center}}
\newcommand{\ed}{\end{tikzcd}}
\newcommand{\bdp}{\begin{center}\begin{tikzpicture}}
\newcommand{\edp}{\end{tikzpicture}\end{center}}
\newcommand{\bi}{\begin{itemize}[leftmargin=*]}
\newcommand{\ei}{\end{itemize}}
\newcommand{\bt}{\begin{tikzpicture}}
\newcommand{\et}{\end{tikzpicture}}
\newcommand{\ba}{\[\begin{aligned}}
\newcommand{\ea}{\end{aligned}\]}
\newcommand{\bp}{\begin{pmatrix}}
\newcommand{\ep}{\end{pmatrix}}
\newcommand{\bsm}{\begin{smallmatrix}}
\newcommand{\esm}{\end{smallmatrix}}
\newcommand{\bv}{\begin{vmatrix}}
\newcommand{\ev}{\end{vmatrix}}
\newcommand{\bb}{\begin{bmatrix}}
\newcommand{\eb}{\end{bmatrix}}
\newcommand{\bB}{\begin{Bmatrix}}
\newcommand{\eB}{\end{Bmatrix}}
\newcommand{\bea}{\begin{enumerate}[leftmargin=*,label=\textnormal{(\alph*)}]}
\newcommand{\ber}{\begin{enumerate}[leftmargin=*,label=\textnormal{(\roman*)}]}
\newcommand{\ben}{\begin{enumerate}[leftmargin=*,label=\textnormal{(\arabic*)}, wide,  labelindent=0pt, labelwidth=*]}
\newcommand{\ee}{\end{enumerate}}

\newcommand{\hlb}{\color{black}}

\makeatletter
\renewcommand{\boxed}[1]{\text{\fboxsep=.2em\fbox{\m@th$\displaystyle#1$}}}
\makeatother

\title[Unimodality of ranks and a proof of Stanton's conjecture]{Unimodality of ranks and a proof of Stanton's conjecture}
\author{Kathrin Bringmann}
\author{Siu Hang Man}
\address{University of Cologne, Department of Mathematics and Computer Science, Weyertal 86-90, 50931 Cologne, Germany}
\email{kbringma@math.uni-koeln.de}
\email{sman1@math.uni-koeln.de}

\address{Charles University, Faculty of Mathematics and Physics, Department of Algebra, Sokolovská 83, 186 00 Praha 8, Czech Republic}
\email{shman@karlin.mff.cuni.cz}

\author{Larry Rolen}
\address{Department of Mathematics, 1420 Stevenson Center, Vanderbilt University, Nashville, TN 37240}
\email{larry.rolen@vanderbilt.edu}
\date{} 

\thanks{The first and the second author have received funding from the European Research Council (ERC) under the European Union’s Horizon 2020 research and innovation programme (grant agreement No. 101001179). The second author was also supported by the Czech Science Foundation GAČR, grant 21-00420M. The third author was supported by a grant from the Simons Foundation (853830) and is also grateful for support from a 2021-2023 Dean's Faculty Fellowship from Vanderbilt University and to the Max Planck Institute for Mathematics in Bonn for its hospitality and financial support.}

\begin{document}
\maketitle
\vspace*{-0.45in}
\begin{abstract}
	Recently, much attention has been given to various inequalities among partition functions. For example, Nicolas, {and later DeSavlvo--Pak,} proved that $p(n)$ is eventually log-concave, and Ji--Zang showed that the cranks are eventually unimodal. This has led to a flurry of recent activity generalizing such results in different directions. At the same time, Stanton recently made deep conjectures on the positivity of certain polynomials associated to ranks and cranks of partitions, with the ultimate goal of pointing the way to ``deeper'' structure refining ranks and cranks. These have been shown to be robust in recent works, which have identified further infinite families of such conjectures in the case of colored partitions. In this paper, we employ the Circle Method to prove unimodality for ranks. As a corollary, we prove Stanton's original conjecture. This points to future study of the positive, integral coefficients Stanton conjectured to exist, hinting at new combinatorial structure yet to be uncovered.
\end{abstract}

\section{Introduction and statement of results}

The study of the sizes and of various inequalities among partitions has a long history, going back to Hardy and Ramanujan's original development of the Circle Method to prove their famous asymptotic for $p(n)$, the number of integer partitions of $n$. Recently, there has been a reinvigoration of this subject in several directions. {Nicolas  proved, and later DeSalvo-Pak independently reproved} \cite{DeSalvoPak,Nicolas} that $p(n)$ is log-concave starting at $n=26$. That is,
\[
p(n)^2\geq p(n-1)p(n+1),\quad \text{ for } n\geq26.
\]
An infinite family of generalizations of log-concavity was conjectured by Chen--Jia--Wang \cite{ChenJiaWang} and proven by Griffin, Ono, Zagier, and the third author \cite{GORZ}. This was in direct analogy with a program of Jensen--Pol\'a on the Riemann Hypothesis, and uncovered a general phenomenon which gave new evidence for it. Another striking multiplicative-additive inequality recently  was proven by Bessenrodt--Ono \cite{BessenrodtOno}
\[
p(a)p(b)\geq p(a+b)\qquad\text{ for }a,b\geq2\ a+b>8.
\]

In this paper we prove further inequalities in the world of partitions, namely for rank counting functions. Although analytic in nature, the motivation behind these inequalities, based on ideas of Stanton, hints at combinatorial structure ``refining'' partition ranks. In conjunction with recent works extending Stanton's conjectures, this suggests that the deeper combinatorial structure Stanton posited is waiting to be discovered. We highlight this in Question~\ref{TheQuestion} below.

To set up this picture, we recall the classical story of ranks and cranks.
Ramanujan \cite{Ramanujan1920,Ramanujan1921} {initiated the arithmetic study of $p(n)$ by stating and proving his famous three congruences}
\begin{align*}
	p(5n+4)&\equiv 0 \pmod 5, & p(7n+5)&\equiv 0 \pmod 7, & p(11n+6)&\equiv 0 \pmod{11}.
\end{align*}
{His proofs relied on $q$-series manipulation. Given the combinatorial nature of the statement, it is natural to ask for a combinatorial proof.} In the search for such a  combinatorial explanation, Dyson \cite{Dyson1944} defined the {\it rank} of a partition to be its largest part minus its number of parts, and conjectured that the partitions of $5n+4$ (resp. $7n+5$) are equidistributed with respect to the rank modulo $5$ (resp. $7$). Dyson's conjecture was proved by Atkin and Swinnerton-Dyer \cite{ASD1954}. Dyson also conjectured the existence of another statistic which he called the crank, which would explain all of Ramanujan's congruences. {More than four decades later}, Andrews and Garvan \cite{AG1988} found such a statistic. For a partition $\lambda$, let $\ell(\lambda)$ be the largest part of $\lambda$, $o(\lambda)$ the number of ones in $\lambda$, and $\mu(\lambda)$ the number of parts in $\lambda$ larger than $o(\lambda)$. The {\it crank} of $\lambda$ is then defined as
\[
	\crank(\lambda) := \begin{cases} \ell(\lambda) & \text{ if } o(\lambda)=0,\\ \mu(\lambda)-o(\lambda) & \text{ if } o(\lambda) > 0.\end{cases}
\]

{Here, we are motivated by a recent search, spearheaded by Stanton, for combinatorial structure refining ranks and cranks.}
While the rank and crank split partitions into classes of equal size, there is no known direct map between these classes. {That is, while ranks and cranks provide a combinatorial explanation for Ramanujan's congruences, the proofs of the rank and crank equidistribution properties are also done via $q$-series, and thus beg for a combinatorial explanation themselves.} In the search of such a map, Stanton proposed several conjectures  in his unpublished notes \cite{Stanton2019} regarding the rank and crank of partitions. Atkin and Swinnerton-Dyer \cite{ASD1954} showed that we have {the following {generating function}}
\begin{equation}\label{eq:rank_zeta}
	R(\z;q) := \sum_{\substack{m\in\Z\\n\ge0}} N(m,n)\z^mq^n =: \sum_{n=0}^\infty \rank_n(\z)q^n = \sum_{n=0}^\infty \frac{q^{n^2}}{(\z q;q)_n\left(\z^{-1}q;q\right)_n}.
\end{equation}
{Here,} $(a;q)_n:=(a)_n:=\prod_{j=0}^{n-1}(1-aq^j)$ ($n\in\N_0\cup\{\infty\}$) is the usual \textit{$q$-Pochhammer symbol}. Meanwhile, Andrews and Garvan \cite{AG1988} proved that the crank generating function
\begin{align}\label{eq:crank_zeta}
	C(\z;q) := \sum_{\substack{m\in\Z\\n\ge0}} M(m,n)\z^mq^n =: \sum_{n=0}^\infty \crank_n(\z)q^n = \prod_{n=1}^\infty \frac{1-q^n}{\left(1-\z q^n\right)\left(1-\z^{-1}q^n\right)}
\end{align}
 gives the {right} crank $M(m,n)$ except for $n=1$ (the correct values are $M(0,1)=1$, and $M(m,1)=0$ for $m\ne0$). 
{We need a  modification of the Laurent polynomials in \eqref{eq:rank_zeta} and \eqref{eq:crank_zeta}.}

\begin{definition*}
	For $n\in\N_0$, $\ell \in \cb{5,7,11}$, and $\beta:= \ell - \frac{\ell^2-1}{24}$,  the {\it{ modified rank}} and {\it{crank}} are
	\begin{align*}
		\rank_{\ell,n}^*(\zeta) &:= \rank_{\ell n+\beta}(\zeta) + \zeta^{\ell n+\beta-2} - \zeta^{\ell n+\beta-1} + \zeta^{-\ell n-\beta+2} - \zeta^{-\ell n-\beta+1},\\
		\crank_{\ell,n}^*(\zeta) &:= \crank_{\ell n+\beta}(\zeta) + \zeta^{\ell n +\beta-\ell} - \zeta^{\ell n+\beta} + \zeta^{-\ell n-\beta+\ell} - \zeta^{-\ell n-\beta}.
	\end{align*}
\end{definition*}

\begin{example*}
	These modifications simply redefine the rank or crank of two partitions each. In the case of the rank, there is always one partition of $n$ of rank $n-1$, namely, $n$. The term $-\z^{\ell n+\b-1}$ deletes this, and assigns it the rank $n-2$. Similarly, the third and fourth monomials in the definition of $\rank_{\ell,n}^*(\z)$ ``reassigns'' the partition $1+1\ldots+1$ with rank $1-n$ to have rank $2-n$. While it may seem strange at first glance, this choice is natural as it turns out it fixes the ``trivial'' lack of unimodality of the number of partitions of $n$ with rank $m$ at the tails. The main result of this paper show that for $n\ge39$, there are no other failures of unimodality, as shown in this table: 
\begin{table}[h!]\label{RanksTable}\caption{Table of $N(m,n)$ for $n=1,2,3,4,5,6,\ldots,38,39$}
\begin{tabular}{ c c c c c c c c c c c }
 &  & & & & 1& & & & & \\ 
 &  & & &1 & 0& 1& & & & \\  
 &  & & 1& 0& 1&0 &1 & & &\\
 &  & 1& 0& 1& 1&1&0 & 1& & \\
  &  1&0& 1& 1& 1&1&1 & 0&1 & \\
   1 &  0&1& 1& 2& 1&2&1 & 1&0 &1 \\
      \vdots &  \vdots&\vdots& \vdots& \vdots& \vdots&\vdots&\vdots & \vdots&\vdots &\vdots \\
      \ldots &  1169&1273& 1331& 1390& 1389&1390&1331 & 1273&1169 &\ldots \\
            \ldots &  1404&1494& 1592& 1627& 1661&1627&1592 & 1494&1404 &\ldots \\

\end{tabular}
\end{table}

\end{example*}

To state Stanton's conjectures on ranks and cranks we let $\Phi_\ell(\z):=1+\z+\ldots+\z^{\ell-1}$ denote the $\ell$-th cyclotomic polynomial, where $\z_\ell:=e^\frac{2\pi i}{\ell}$. 

\begin{conjecture}[Stanton]\label{cnj:Stanton}
	Let $n\in\N_0$.
	\begin{enumerate}[label=\rm{(\arabic*)},wide,labelwidth=!,labelindent=0pt]
		\item The following expressions are Laurent polynomials with non-negative coefficients:
		\[
			\frac{\rank_{5,n}^*(\zeta_\ell)}{\Phi_5(\zeta_\ell)} \quad \text{ and } \quad \frac{\rank_{7,n}^*(\zeta_\ell)}{\Phi_7(\zeta_\ell)}.
		\]
		
		\item The following expressions are Laurent polynomials with non-negative coefficients:
		\[
			\frac{\crank_{5,n}^*(\zeta_\ell)}{\Phi_5(\zeta_\ell)}, \quad \quad \frac{\crank_{7,n}^*(\zeta_\ell)}{\Phi_7(\zeta_\ell)} \quad \text{ and } \quad \frac{\crank_{11,n}^*(\zeta_\ell)}{\Phi_{11}(\zeta_\ell)}. 
		\]
	\end{enumerate}
\end{conjecture}
These quotients being Laurent polynomials is equivalent to the equidistribution of ranks and cranks. In \cite{BGRT2021}, the first and the third author, along with Gomez and Tripp reduced the conjecture to a problem on the unimodality of the modified rank and crank. For the crank, Ji and Zang \cite{JZ2021} proved a related result, via a careful analysis of the generating function. 

The main goal of this paper is to prove Conjecture \ref{cnj:Stanton}. To state our methods, we first recall the previous proof of part (1), and then outline our methods for proving part (2). It turns out that the key phenomenon beyond Stanton's Conjecture is the eventual near-unimodality of ranks and cranks. 
{In the case of cranks, let $M(m,n)$ denote the number of partitions of $n$ with crank $m$. Eventual unimodality in this case was recently proven} by Ji and Zang \cite{JZ2021} {(as $M(m,n)=M(-m,n)$, unimodality is the same as monotonicity starting at $m=0$).}

\begin{theorem}[\cite{JZ2021}, Theorem 1.7]\label{thm:JZ}
For $n\ge 44$ and $0\le m \le n-2$, we have
\[
M(m,n) \ge M(m+1,n).
\]
\end{theorem}

Using \Cref{thm:JZ}, the first  and the third author, Gomez, and Tripp \cite{BGRT2021} established Stanton's conjecture for the crank.

\begin{theorem}[\cite{BGRT2021}, Theorem 1.3]
	\Cref{cnj:Stanton} {\rm(2)} is true.
\end{theorem}

To prove Stanton's conjecture for the rank, we need a result analogous to \Cref{thm:JZ} for the rank, which was conjectured in \cite{BGRT2021}. The methods used to prove Theorem~\ref{thm:JZ} were the judicious use of clever $q$-series manipulations and comparisons. These methods do not, however, generalize to other situations. Here, we give a more general analytic proof based on the Circle Method. 

The monotonicity result we aim to show here is the following, where $N(m,n)$ denote the number of partitions of $n$ with rank $m$.

\begin{conjecture}\label{cnj:rank_unimodality}
For $n\ge 39$ and $0\le m \le n-3$, we have
\[
N(m,n)\ge N(m+1,n).
\]
\end{conjecture}
As shown in \cite{BGRT2021}, this is the key analytic result needed to prove Stanton's Conjecture.
\begin{theorem}[\cite{BGRT2021}, Theorem 1.5]
\Cref{cnj:rank_unimodality} implies \Cref{cnj:Stanton} {\rm(1)}.
\end{theorem}

\Cref{cnj:rank_unimodality} is known to be true for $n$ sufficiently large; in fact, Zhou \cite{Zhou2021} showed that $\cb{N(m,n)}_{\vb{m}\le n-73}$ is log-concave. However, such proofs are quite complicated and required different arguments depending on the size of $m$, and no explicit bounds are known in the literature. The main result of the paper is the following.

\begin{theorem}\label{thm:main}
	\Cref{cnj:rank_unimodality}, and hence \Cref{cnj:Stanton} is true.
\end{theorem}

{Returning to Stanton's original goals for his conjecture, this suggests the following. 
\begin{question}\label{TheQuestion}
What is the combinatorial meaning of the positive numbers occurring in Conjecture~\ref{cnj:Stanton}? What does this suggest about bijections between ranks and cranks mod $5,7,11$?
\end{question}
This question is further bolstered by recent work in \cite{RolenTrippWagner} which gave a procedure for finding infinite families of crank-type functions explaining ``most'' Ramanujan-type congruences for all $k$-colored partition functions, and in \cite{BGRT2021} which reworked these families to give new families which explained the same Ramanujan-type congruences but also satisfy Stanton-type conjectures. This suggests that the phenomenon is very general, and we plan to address this in the future. 
}

We outline the strategy for proving \Cref{thm:main}. Our proof for  $m\ll n^{\frac 34}$ is based on Zhou's work on the eventual log-concavity of the rank \cite{Zhou2021}, which in turn is based on Wright's Circle Method. Typical for such approaches, this only shows that there exists $n_0\in\N$ (which may be very large), such that the statement is true for $n\ge n_0$. In this paper, we fine-tune the arguments in \cite{Zhou2021}, establish strong error bounds, so that the constant $n_0$ is sufficiently small to make explicit checking of the remaining cases feasible. For $m\gg n^{\frac 12}$, we have another argument, which is based on the convexity of the partition function $p(n)$. Combining these arguments, we find an explicit $n_0$ such that the statement is true for $n\ge n_0$, and check the remaining cases by computer. 

{The paper is organized as follows. In Section~\ref{Prelim} we give some preliminaries on asymptotics, particularly for the reciprocal of $(q)_\infty$, the Euler--Maclaurin summation formula, explicit estimates for Bessel functions, and resulting bound for partitions and partition differences. Using these analytic estimates, in Section~\ref{section:cm}, we set up the Circle Method and bound minor and major arc contributions. To prove our main result, we require extensive computer computations, which are summarized in a case-by-case basis in Section~\ref{section:num_com}, culminating in the proof of Theorem~\ref{thm:m_small} which proves Theorem~\ref{thm:main} for  small $m$ (specifically, asymptotically up to about $n^{\frac34}$) for $n$ sufficiently large. We conclude with the proof of Theorem~\ref{thm:main}, that is, Stanton's Conjecture,  in Section~\ref{FinalSection}.

}

\section{Preliminaries}\label{Prelim}
\subsection{Generating functions}

We {require the} following generating function $N(m,n)$ (see \cite[(2.12)]{ASD1954}),
\[
	\Nc_m(q) := \sum_{n=0}^\infty N(m,n)q^n = \frac{1}{(q;q)_\infty}\sum_{r=1}^\infty (-1)^{r+1}q^{\frac{r(3r-1)}{2}+mr}\left(1-q^r\right).
\]
Since $N(-m,n)=N(m,n)$ for $m\in\Z$, we may assume that $m\in\N_0$. As we aim to prove monotonicity, we are interested in $N(m,n)-N(m+1,n)$, {which has the generating function}
\begin{equation}\label{eq:diff_gf}
	\mc N_m(q) - \mc N_{m+1}(q) = \sum_{n=0}^\infty \rb{N(m,n) - N(m+1,n)} q^n = \frac{H_m(q)}{(q;q)_\infty},
\end{equation}
where
\[
	H_m(q) := \sum_{r=1}^\infty (-1)^{r+1}q^{\frac{3r^2}{2}+\left(m+\frac12\right)r}\left(q^{-\frac r2}-q^\frac r2\right)^2. 
\]

\subsection{Euler--Maclaurin summation}
{To produce further asymptotics in the Circle Method, we require the Euler--Maclaurin summation. In general, it gives an exact formula for  infinite sums in terms of an associated integral. 
To this end, let} $f\colon\R\to\R$ be a differentiable function, such that $f$ and its derivative $f'$ have rapid decay, {and let} $N\in\Z$. A simple {case of } the Euler--Maclaurin summation formula, {sufficient for our purposes,} states that
\begin{equation}\label{eq:EM}
	\sum_{j=N}^\infty f(j) = \int_N^\infty f(u) du + \frac{f(N)}{2} + \int_N^\infty f'(u)B_1(u-\flo{u}) du,
\end{equation}
where $B_1(u)=u-\frac12$ is the first Bernoulli polynomial.\footnote{A more detailed exposition on the Euler--Maclaurin summation formula can be found in \cite[p. 319]{Zagier2006}.} 

\subsection{Estimates and expansions for Bessel functions}
{We also require tight estimates for $I$-Bessel functions. Specifically, let $u>0$ and $\kappa\ge 0$. Then we approximate $I_{-\k-1}(u)$, the $I$-Bessel function of order $-\k-1$, by comparing it to the following integral
\[
	\mc I_\kappa(u) := \frac{1}{2\pi i} \int_{1-i}^{1+i} w^{-\kappa-1} \exp\rb{\frac u2 \rb{w+\frac1w}} dw
\]
This gives us the following, whose proof we omit as it is standard, following from the  integral representation of the $I$-Bessel function (see e.g. \cite{Arken1985}).}

\begin{lemma}\label{lem:Bessel_segment}
{For $u\in\R^+$ and $\k< 0$, we have}
\[
		| I_{k}(u)-\mc I_\k(u) |\le \frac{2^{\frac{1-\k}{2}}}{\pi} \rb{\frac{2^{-\k-1}\Gamma(-\k)}{u^{-\k}} + \exp\rb{\frac{3u}{4}}},
	\]
\end{lemma}

{We also require some facts on the asymptotic expansions of the $I$-Bessel functions. First, we recall standard expansions. }
Let $N,M\in\N_0$, $s\in\C$. Then we have for $w\in\C$  
\begin{equation*}\label{eq:IBes_exp}
	I_s(w) = \frac{e^w}{\sqrt{2\pi w}} \rb{\sum_{j=0}^{N-1} (-1)^j \frac{a_j(s)}{w^j} + \delta_N(s,w)} - ie^{-\pi i s} \frac{e^{-w}}{\sqrt{2\pi w}} \rb{\sum_{j=0}^{M-1} \frac{a_j(s)}{w^j} + \gamma_M(s,w)}
\end{equation*}
(see \cite[(exercise 7.13.2)]{Olver1997}),
where 
\ba
	a_0(s) := 1,\qquad a_j(s) := \frac1{8^jj!}{\prod_{k=1}^j \left(4s^2-(2k+1)^2\right)}.
\ea
Moreover, the error terms $\d_N$ and $\g_M$ satisfy the following bounds: 
\begin{align*}
	|\g_M(s,w)| &\le 2\exp\left(\vb{\frac{s^2-\frac14}{w}}\right)\vb{\frac{a_M(s)}{w^M}},\quad |\d_N(s,w)| \le \frac{2\pi^\frac12\Ga\left(\frac N2+1\right)}{\Ga\left(\frac N2+\frac12\right)} \exp\left(\frac\pi2\vb{\frac{s^2-\frac14}{w}}\right)\vb{\frac{a_N(s)}{w^N}}.
\end{align*}
For later use, we consider for $\nu\in\N_0$ the following linear combination of $I$-Bessel functions:
\[
	I_s^{[\nu]}(w) := \sum_{r=0}^\nu \binom \nu r (-1)^{\nu+r} I_{s-r}(w).
\]
Then $I_s^{[\nu]}(w)$ has an asymptotic expansion of the following form:
\begin{equation}\label{eq:Besselv_asymp}
	I_s^{[\nu]}(w) = \hspace{-1mm}\frac{e^w}{\sqrt{2\pi w}} \rb{\sum_{j=0}^{N-1} (-1)^j \frac{a_j^{[\nu]}(s)}{w^j} + \delta_N^{[\nu]}(s,w)} \hspace{-1mm}- ie^{-\pi i s} \frac{e^{-w}}{\sqrt{2\pi w}} \rb{\sum_{j=0}^{M-1} \frac{a_j^{[\nu]}(s)}{w^j} + \gamma_M^{[\nu]}(s,w)},\hspace{-2mm}
\end{equation}
where
\ba
a_j^{[\nu]}(s) &:= \sum_{r=0}^\nu \binom \nu r (-1)^{\nu+r} a_j(s-r), \quad
	\vb{\gamma_M^{[\nu]}(s,w)} \le \sum_{r=0}^\nu \binom \nu r \vb{\gamma_M(s-r,w)},\\ \vb{\delta_N^{[\nu]}(s,w)} &\le \sum_{r=0}^\nu \binom \nu r \vb{\delta_N(s-r,w)}.
\ea
{Finally, a direct calculation gives the following.}

\begin{proposition}\label{prp:Besselv_asymp}
	Let $\k\ge0$ and $\nu\in\N_0$. We have that $a_j^{[\nu]}(\k)=0$ for $j<\frac\nu2$.
\end{proposition}

\subsection{Estimates for partitions}

{We first give basic estimates for $(q;q)_\infty^{-1}$. A direct calculation, using the modularity of the Dedekind $\eta$-function (see 5.8.1 of \cite{CS}) and bounds for $p(n)$ (see \cite[Theorem 14.5]{Apostol1976}) gives {the following.}\footnote{The explicit constants given here (and for the rest of the paper) do not need to be optimal.}

\begin{lemma}\label{lem:Z3.1}
	Let $Z=x+iy$, where $x,y\in\R$, $0<x\le\frac{1}{10}$, $0\le|y|\le\pi$. {Further let} $q:=e^{-Z}$.
	\begin{enumerate}[leftmargin=*,label=\rm{(\arabic*)}]
		\item For $|y|\le x$ {we have}, with $|C|\leq1$
		\begin{equation*}
			\frac{1}{(q;q)_\infty} = \sqrt{\frac{Z}{2\pi}} \exp\rb{-\frac{Z}{24}+\frac{\pi^2}{6Z}} + C.
		\end{equation*}

		\item For $x\le|y|\le\pi$, we have
		\begin{equation}\label{eq:Z3.1_claim2}
			\vb{\frac{1}{(q;q)_\infty}} \le 2\sqrt{x}\exp\left(\frac{\pi^2}{12x}\right).
		\end{equation}
	\end{enumerate}
\end{lemma}

{We also require an approximation for $p(n)$ with bounded error}. Using Rademacher's convergent series expansion for the partition function $p(n)$ \cite{Rademacher1937} one can show the following asymptotic formula (see \cite[Lemma 5.6]{Zhou2021} for a non-explicit version).

\begin{lemma}\label{lem:partition_asymp}
	For all $n\in\N$, {there exists a constant $|\Cc|\le\frac18$ for which}
		\begin{equation*}
		p(n) = \frac{1}{4\sqrt{3}}\frac{e^{\pi\sqrt{\frac23\left(n-\frac{1}{24}\right)}}}{n-\frac{1}{24}} - \frac{1}{4\sqrt{2}\pi}\frac{e^{\pi \sqrt{\frac23\left(n-\frac{1}{24}\right)}}}{\left(n-\frac{1}{24}\right)^\frac32} + \frac{\Cc e^{\frac\pi2\sqrt{\frac23\left(n-\frac{1}{24}\right)}}}{n-\frac{1}{24}}.
	\end{equation*}
\end{lemma}

Note that \Cref{lem:partition_asymp} implies the following convenient upper bound for $n\in\N$:
\begin{equation}\label{eq:partition_ubound}
	p(n) \le \frac{e^{\pi\sqrt{\frac23\left(n-\frac{1}{24}\right)}}}{4\sqrt{3}\left(n-\frac{1}{24}\right)}.
\end{equation}
{Finally, for} $n\in\N_0$, $r\in\N$, we require strong estimates on the following ``second order'' differences
\begin{equation}\label{eq:wp_def}
	P(n,r) := p(n)-2p(n+r)+p(n+2r). 
\end{equation}
{Then one can use Lemma~\ref{lem:partition_asymp} directly to give the following key estimate.}

\begin{lemma}\label{lem:sandwich} 
Let $n\ge 100$, $r\in\N$. Then we have
\[
\frac{r^2 e^{\pi\sqrt{\frac 23\rb{n-\frac{1}{24}}}}}{6\rb{n+2r-\frac 1{24}}^2}  \le P(n,r) \le \frac{\pi^2 r^2 e^{\pi\sqrt{\frac 23 \rb{n+2r-\frac 1{24}}}}}{24\sqrt{3}\rb{n-\frac 1{24}}^2}. 
\]
\end{lemma}
\section{The Circle Method}\label{section:cm}
We next recall a version of the Circle Method that we require for the proof of our main result.
\subsection{The setup}

For $n\in\N$, we let
\begin{equation}\label{eq:bL_def}
	\beta_n := \frac{\pi}{\sqrt{6\left(n-\frac 1{24}\right)}}, \hspace{1cm} \mathcal{C}_n := \rb{n-\frac 1{24}} \beta_n = \frac{\pi}{\sqrt 6} \sqrt{n-\frac 1{24}}.
\end{equation}

Denote by $\Cc$ the circle with radius $e^{-\b_n}$, centered at the origin, traversed counter-clockwise exactly once. Cauchy's Theorem gives 
\begin{equation}\label{eq:Cauchy_int}
	N(m,n) - N(m+1,n) = \frac{1}{2\pi}\int_{-\pi}^\pi \frac{H_m\left(e^{-Z}\right)}{\left(e^{-Z};e^{-Z}\right)_\infty}e^{nZ} dy. 
\end{equation}
We split the integral into two parts: the {\it major arc integral}
\begin{equation}\label{eq:maj_arc_int}
	\mf M := \frac{1}{2\pi} \int_{\vb{y}\le \beta_n} \frac{H_m\left(e^{-Z}\right)}{(e^{-Z};e^{-Z})_\infty} e^{nZ} dy,
\end{equation}
and the {\it minor arc integral}
\begin{equation}\label{eq:min_arc_int}
	\mf m := \frac{1}{2\pi} \int_{\beta_n \le \vb{y}\le \pi} \frac{H_m\left(e^{-Z}\right)}{(e^{-Z};e^{-Z})_\infty} e^{nZ} dy.
\end{equation}
{We bound the minor arc contributions and approximate the major arc contributions, which give our main terms.}

\subsection{Minor arc bounds}

We begin by expanding
\[
	H_m\left(e^{-Z}\right) = \sum_{r=1}^\infty (-1)^{r+1} \left(e^{-\left(\frac{3r^2}{2}+\left(m-\frac12\right)r\right)Z} - 2e^{-\left(\frac{3r^2}{2}+\left(m+\frac12\right)r\right)Z} + e^{-\left(\frac{3r^2}{2}+\left(m+\frac32\right)r\right)Z}\right).
\]
Since $m\ge 0$, we have $\frac{3 r^2}{2}+ (m+j)r \ge \frac{3 r^2}{2} + jr \ge \frac 12 (r-1)^2$ for $j \in \{-\frac 12, \frac 12, \frac 32\}$ and $r\ge 1$. Hence
\[
	\vb{H_m\left(e^{-Z}\right)} \le \sum_{r=1}^\infty \left(e^{-\left(\frac{3r^2}{2}+\left(m-\frac12\right)r\right)\b_n} + 2e^{-\left(\frac{3r^2}{2}+\left(m+\frac12\right)r\right)\b_n} + e^{-\left(\frac{3r^2}{2}+\left(m+\frac32\right)r\right)\b_n}\right) \le 4\sum_{r\ge0} e^{-\frac{\b_n r^2}{2}}. 
\]
It is not hard to see that we may bound
\begin{equation}\label{eq:Hmez_bound}
	\vb{H_m\left(e^{-Z}\right)} \le 2\sqrt{\frac{2\pi}{\b_n}} + 4.
\end{equation}
Using this, we now turn to bounding the minor arc integral $\mf m$.

\begin{proposition}\label{prp:min_arc_bound}
	For  $\b_n\le\frac{1}{10}$, we have
	\[
		\vb{\mf m} \le 13 e^{\frac {3 \mc C_n}2}. 
	\]
\end{proposition}
\begin{proof}
By \eqref{eq:Z3.1_claim2}, for $\beta_n \le \frac{1}{10}$ and $\beta_n \le \vb{y} \le \pi$ we have 
\[
	\vb{\frac{1}{\left(e^{-Z};e^{-Z}\right)_\infty}} \le 2\sqrt{\b_n}e^\frac{\pi^2}{12\b_n}.
\]
Hence
\[
	|\mf m| \le 4\left(\sqrt{2\pi}+2\sqrt{\b_n}\right)e^{\frac{\pi^2}{12\b_n}+n\b_n}. 
\]
Now a direct calculation gives the claim.
\end{proof}

\subsection{Major arc bounds}

Next we consider the major arc integral $\mf M$. Again we assume $\b_n\le\frac{1}{10}$. On the major arc, we have $|y|\le\b_n$, so we can use \Cref{lem:Z3.1} and obtain
\begin{equation*}
	\mf M = \frac{1}{2\pi} \int_{\vb{y}\le \beta_n} \frac{H_m\rb{e^{-Z}}e^{nZ}}{(e^{-Z};e^{-Z})_\infty} dy = \frac{1}{2\pi} \int_{\vb{y}\le \beta_n} \rb{\sqrt{\frac{Z}{2\pi}} e^{-\frac{Z}{24}+\frac{\pi^2}{6Z}} + C} H_m\rb{e^{-Z}}e^{nZ} dy,
\end{equation*}
for some $|C|\le1$. Let
\begin{equation}\label{eq:M1_def}
	\mc M^{[1]} := \frac{1}{2\pi} \int_{\vb{y}\le \beta_n} \sqrt{\frac{Z}{2\pi}} e^{\left(n-\frac 1{24}\right)Z+\frac{\pi^2}{6Z}} H_m\rb{e^{-Z}} dy,\quad \mc E^{[1]} := \mf M - \mc M^{[1]}.
\end{equation}
Then it follows from \eqref{eq:Hmez_bound} that for $\b_n\le\frac{1}{10}$, {we have}
\begin{equation}\label{eq:E1_display}
	\vb{\mc E^{[1]}} \le \frac{1}{2\pi}\int_{|y|\le\b_n} \vb{H_m\left(e^{-Z}\right)e^{nZ}} dy \le \frac{2\b_n}{\pi}\left(\sqrt{\frac{2\pi}{\b_n}}+2\right)e^{n\b_n} \le e^{\Cc_n}.
\end{equation}
 To give an estimate to $\Mc^{[1]}$, we need a bound for $H_m(e^{-Z})$. Note that $|Z|\le\sqrt{2}\b_n$. For any $L\in\mathbb N$, we approximate $H_m(e^{-Z})$ by a sum of $L-1$ terms. Specifically, we define 
 
\begin{align}\nonumber
	\Mc_{\ell,m}(Z) &:= \frac{2Z^{2\ell}}{(2\ell)!}\sum_{r=1}^\infty (-1)^{r+1}r^{2\ell}e^{-\left(\frac{3r^2}{2}+\left(m+\frac12\right)r\right)Z},\\
	\label{eq:EHm_def}
	\Ec_{L,m}(Z) &:= H_m\left(e^{-Z}\right) - \sum_{\ell=1}^{L-1} \Mc_{\ell,m}(Z).
\end{align}
Moreover, we set
\begin{equation}\label{eq:Cnew01_def}
	C_1(L,x_0) := \frac{2^{L+1} e^{\frac{x_0}2}}{(2L-1)!} \rb{\frac{\Gamma\rb{L+\frac 12}}2 + e^{-L} L^L \sqrt{x_0}}.
\end{equation}

{In terms of these quantities, we obtain the following estimate for the error term $\Ec_{L,m}(Z)$.}
\begin{lemma}\label{lem:ELmz_bound}
	Let $x_0\in\R^+$. Then we have, for $\b_n\le x_0$,
	\begin{equation*}
		\vb{\mc E_{L,m}(Z)} \le C_1(L,x_0)\b_n^{L-\frac12}e^{-\left(m+\frac12\right)\b_n}.
	\end{equation*}
\end{lemma}

\begin{proof}
	Using Taylor's Theorem, we obtain for $L\in\N$ 
	\begin{equation*}
		\left(e^\frac w2-e^{-\frac w2}\right)^2 = w^{2\ell}\sum_{\ell=1}^{L-1} \frac{2}{(2\ell)!} + \frac{1}{(2L-1)!}\int_0^w (w-u)^{2L-1}\left(\pd{}{u}\right)^{2L}\left(e^{-\frac u2}-e^{\frac u2}\right)^2 du.
	\end{equation*}
	Letting $w=rZ$ yields the bound
	\begin{equation*}
		\vb{\left(e^\frac{rZ}{2}-e^{-\frac{rZ}{2}}\right)^2 - 2\sum_{\ell=1}^{L-1} \frac{(rZ)^{2\ell}}{(2\ell)!}} \le \frac{2}{(2L-1)!}(r|Z|)^{2L}e^{r\b_n}.
	\end{equation*}
	Note that
	\begin{equation}\label{eq:EHm_extract}
		|\Ec_{L,m}(Z)| \le \frac{2^{L+1}\b_n^{2L}e^{-\left(m+\frac12\right)\b_n}}{(2L-1)!}\sum_{r=1}^\infty r^{2L}e^{-\left(\frac{3r^2}{2}+\left(m+\frac12\right)(r-1)\right)\b_n+r\b_n}. 
	\end{equation}
	Next we estimate the sum in \eqref{eq:EHm_extract}. Since $m\ge0$, we have $\frac{3r^2}{2}+(m+\frac12)(r-1)-r\ge\frac{3r^2}{2}+\frac12(r-1)-r\ge r^2-\frac12$ for $r\ge1$. We deduce that
	\begin{equation}\label{eq:EHm_quadexp}
		\sum_{r=1}^\infty r^{2L}e^{-\left(\frac{3r^2}{2}+\left(m+\frac12\right)(r-1)\right)\b_n+r\b_n} \le e^\frac{\b_n}2\sum_{r=1}^\infty r^{2L}e^{-r^2\b_n}.
	\end{equation}
	The Euler--Maclaurin summation formula \eqref{eq:EM} gives (using that $L>0$)
	\begin{equation}\label{eq:EHm_EMF}
		\sum_{r=1}^\infty r^{2L}e^{-r^2\b_n} = \int_0^\infty u^{2L}e^{-\b_n u^2} du + \int_0^\infty \pd{}{u} \left(u^{2L}e^{-\b_nu^2}\right)B_1(u-\flo{u}) du.
	\end{equation}
	For the first integral in \eqref{eq:EHm_EMF}, we evaluate (see \cite[3.326.2]{GR2007})
	\begin{equation}\label{eq:EHm_EMF1}
		\int_0^\infty u^{2L} e^{-\beta_nu^2} du = \frac{\Gamma\rb{L+\frac 12}}2 \beta_n^{-L-\frac 12}. 
	\end{equation}
	Since $\vb{B_1(u-\lfloor u\rfloor)} \le \frac 12$, the absolute value of the second integral in \eqref{eq:EHm_EMF} is bounded by
	\begin{equation}\label{eq:EHm_EMF2}
		\frac12\int_0^\infty \vb{\frac{\del}{\del u} \left(u^{2L}e^{-\b_nu^2}\right)} du = \left(\frac{L}{e\b_n}\right)^L,
	\end{equation}
	using the Fundamental Theorem of Calculus. Plugging \eqref{eq:EHm_EMF1} and \eqref{eq:EHm_EMF2} into \eqref{eq:EHm_EMF}, we obtain
	\[
		\sum_{r=1}^\infty r^{2L}e^{-r^2\b_n} \le \frac{\Ga\left(L+\frac12\right)}{2}\b_n^{-L-\frac12} + \left(\frac{L}{e\b_n}\right)^L.
	\]
	Putting this into \eqref{eq:EHm_quadexp} then yields
	\[
		\sum_{r=1}^\infty r^{2L}e^{-\left(\frac{3r^2}{2}+\left(m+\frac12\right)(r-1)\right)\b_n+r\b_n} \le e^\frac{\b_n}{2}\left(\frac{\Ga\left(L+\frac12\right)}{2}\b_n^{-L-\frac12}+\left(\frac{L}{e\b_n}\right)^L\right).
	\]
	Finally, inserting this into \eqref{eq:EHm_extract} yields 
	\[
		\vb{\mc E_{L,m}(Z)} \le \frac{2^{L+1} \beta_n^{2L} e^{-\left(m+\frac 12\right)\beta_n+\frac{\beta_n}2}}{(2L-1)!} \rb{\frac{\Gamma\rb{L+\frac 12}}2 \beta_n^{-L-\frac 12} + \rb{\frac{L}{e\beta_n}}^L}.
	\]
	This gives the claim.
\end{proof}

{Thus, we may estimate $H_m(e^{-Z})$ by  $\sum_{\ell=1}^{L-1} \Mc_{\ell,m}(Z)$ up to a known error bound. }
We next estimate the terms $\Mc_{\ell,m}(Z)$. We first recall a result from \cite{LZ2022}.

\begin{proposition}[{\cite[Proposition 2.3]{LZ2022}}]\label{prp:LZ2.3}
	Let $w,Z\in\C$, $\re(Z)>0$, $w\not\in 2\pi i(\Z+\frac 12)$, $\ell\in\N_0$, and $K\in\N$. Then we have
	\[
		\sum_{r=1}^\infty (-1)^{r+1} r^\ell e^{-r^2Z-rw} = (-1)^\ell \sum_{k=0}^{K-1} \frac{(-Z)^k}{k!} \left(\frac{\partial}{\partial w}\right)^{\ell+2k} {\frac{1}{1+e^w}} + Z^K R_K(\ell;w,Z),
	\]
	where
	\begin{equation}\label{eq:LZ2.3_RK}
		R_K(\ell;w,Z) := \frac{(-1)^{\ell+K}}{(K-1)!}\int_0^1 (1-t)^{K-1}\left(\pd{}{w}\right)^{\ell+2K} \sum_{r=1}^\infty (-1)^{r+1}e^{-r^2Zt-rw} dt.
	\end{equation}
\end{proposition}

We would like to apply \Cref{prp:LZ2.3} to estimate $\Mc_{\ell,m}(Z)$, following the proof of \cite[Proposition 2.6]{LZ2022}. To this end, we show the following.

\begin{proposition}\label{prop:Cnew11_def}
	Assuming the notation above, we have
	\begin{multline*}
		\vb{\Mc_{\ell,m}(Z)-\frac{2}{(2\ell)!}\sum_{k=0}^{K-1} \frac{\left(-\frac32\right)^kZ^{2\ell+k}}{k!}\left[\left(\pd{}{w}\right)^{2\ell+2k} \frac{1}{1+e^w}\right]_{w=\left(m+\frac12\right)Z}}\\
		\le C_2(\ell,K,J,x_0,\vr)\b_n^{2\ell+K}e^{-\left(m+\frac12\right)\b_n},
	\end{multline*}
	 where the constant $C_2$ is explicitly constructed in the proof.
\end{proposition}

{Before beginning the proof, we  also require basic facts about difference operator. To this end for $j\in\N_0$, let $\Delta^j$ be the {\it $j$-th order forward difference operator}: 
\begin{align*}
	\De^0(h(r)) := h(r),\quad \De(h(r)) := h(r+1) - h(r),\quad
	\De^{j+1}(h(r)) := \De^{j-1}(h(r+1)) - \De^{j-1}(h(r)).
\end{align*}
For any arithmetic function $h$, we have the generating function.
\begin{lemma}[{\cite[Lemma 2.4]{LZ2022}}]\label{lem:LZ2.4}
	Let $z\in\C\setminus\{1\}$, and let $h\colon\N_0\to\C$ satisfy $\sum_{r=0}^\infty|h(r)z^r|<\infty$. Then we have, for all $J\in\N$,
	\[
		\sum_{r=0}^\infty h(r)z^r = \sum_{j=0}^{J-1} \frac{\De^j(h(0))z^j}{(1-z)^{j+1}} + \frac{z^J}{(1-z)^J}\sum_{r=0}^\infty \De^J(h(r))z^r.
	\]
\end{lemma}

\begin{proof}[{Proof of Proposition~\ref{prop:Cnew11_def}}]
	By \Cref{prp:LZ2.3}, we have
	\begin{multline}\label{eq:RK_applied}
		\sum_{r\ge 1} (-1)^{r+1} r^{2\ell} e^{-\rb{\frac{ 3 r^2}{2}+\left(m+\frac 12\right)r}Z} \\
		= \sum_{k=0}^{K-1} \frac{\left(-\frac {3Z}{2}\right)^k}{k!} \left[\left(\frac{\partial}{\partial w}\right)^{2\ell+2k}  {\frac{1}{1+e^w}} \right]_{w = \left(m+\frac 12\right)Z}+ \rb{\frac{3Z}{2}}^K R_K\rb{2\ell;\rb{m+\frac 12}Z,\frac{3Z}{2}}.
	\end{multline}
	Now we bound $R_K(2\ell; bZ,Z)$ for $b\ge 0$, and $Z=x+iy$ with $x\ge 0$, and $\vb{y}\le x$. We start with some technical preparations. For $0\leq x\leq x_0$, $r\in\N$, define $E_Z(r):=e^{-r^2Z}$, let $J\in\N$, and $x_0\le1$. {For $0\leq x\leq x_0$, we would like to bound}
	\begin{equation*}\label{eq:DeltaE_def}
		\De^J(E_Z(r)) = (-1)^J\sum_{j=0}^J (-1)^j\binom Jje^{-(r+j)^2Z}.
	\end{equation*}
	 We do so in two different ways. First assume that $J$ is even and observe that
	\begin{equation}\label{eq:Cnew02_def2}
		\vb{\De^J(E_Z(r))} = e^{-r^2x}\vb{\sum_{j=0}^J (-1)^j\binom Jje^{-\left(2rj+j^2\right)Z}} \le C_{2a}(J,x_0)e^{-r^2x},
	\end{equation}
	where $C_{2a}(J,x_0)$ is a constant such that for all $r\ge0$ and $Z=x+iy$ with $0\le x\le x_0$ and $|y|\le x$
	\begin{equation}\label{eq:Cnew02_def}
		\vb{\sum_{j=0}^J (-1)^j\binom Jje^{-\left(2rj+j^2\right)Z}} \le C_{2a}(J,x_0).
	\end{equation}

	Next we give another bound, which works better if $r$ is small. Note that
	\begin{equation}\label{eq:DeltaE_expand}
		e^{r^2Z}\De^J(E_Z(r)) = \sum_{j=0}^J (-1)^j \binom Jj e^{-\left(2rj+j^2\right)Z}.
	\end{equation}
	By Taylor's Theorem, we obtain
	\begin{equation}\label{eq:DeltaE_Taylor_error}
		\vb{e^{-j^2Z}-\sum_{\nu=0}^{\frac J2-1} \frac{\left(-j^2Z\right)^\nu}{\nu!}} \le \frac{\left(j^2|Z|\right)^\frac J2}{\left(\frac J2\right)!} \le \frac{2^\frac J4j^Jx^\frac J2}{\left(\frac J2\right)!}.
	\end{equation}
	We apply \eqref{eq:DeltaE_Taylor_error} to \eqref{eq:DeltaE_expand} and get 
	\[
		\vb{e^{r^2Z}\De^J(E_Z(r))-\sum_{j=0}^J (-1)^j\binom Jje^{-2rjZ}\sum_{\nu=0}^{\frac J2-1} \frac{(-j^2Z)^\nu}{\nu!}} \le \frac{2^\frac J4x^\frac J2}{\left(\frac J2\right)!}\sum_{j=0}^J \binom Jjj^J.
	\]
	We deduce that
	\begin{align}\nonumber
		\left|e^{r^2Z}\De^J(E_Z(r))\right| &\le \left|\sum_{\nu=0}^{\frac J2-1} \frac{(-Z)^\nu}{\nu!}\sum_{j=0}^J (-1)^j\binom Jjj^{2\nu}e^{-2rjZ}\right| + \frac{2^\frac J4x^\frac J2}{\left(\frac J2\right)!}\sum_{j=0}^J \binom Jj j^J\\
		\label{eq:DeltaE_Taylor}
		&\le \sum_{\nu=0}^{\frac J2-1} \frac{2^\frac\nu2x^\nu}{\nu!}\left|\left[\left(\pd{}{w}\right)^{2\nu} \left(1-e^{-w}\right)^J\right]_{w =2rZ}\right| + \frac{2^\frac J4x^\frac J2}{\left(\frac J2\right)!}\sum_{j=0}^J \binom Jjj^J.
	\end{align}
	For $0\le\nu<\frac J2$, choose $C_{2b}(J,\nu)$ such that for all $w\in\C$ with $\re(w)\ge0$ and $|\im(w)|\le\re(w)$
	\begin{equation}\label{eq:Cnew03_def}
		\vb{w^{2\nu-J}\left(\frac{\del}{\del w}\right)^{2\nu} \left(1-e^{-w}\right)^J} \le C_{2b}(J,\nu).
	\end{equation}

	Plugging \eqref{eq:Cnew03_def} into \eqref{eq:DeltaE_Taylor} gives, using that $|Z|\le\sqrt{2}x$,
	\begin{align}\nonumber
		\vb{e^{r^2Z}\De^J(E_Z(r))} &\le \sum_{\nu=0}^{\frac J2-1} C_{2b}(J,\nu)\frac{2^\frac\nu2x^\nu}{\nu!}|2rZ|^{J-2\nu} + \frac{2^\frac J4x^\frac J2}{\left(\frac J2\right)!}\sum_{j=0}^J \binom Jjj^J\\
		\nonumber
		&\le \sum_{\nu=0}^{\frac J2-1} C_{2b}(J,\nu)\frac{2^{\frac{3J}{2}-\frac{5\nu}{2}}}{\nu!}r^{J-2\nu}x^{J-\nu} + \frac{2^\frac J4x^\frac J2}{\left(\frac J2\right)!}\sum_{j=0}^J \binom Jjj^J\\
		\label{eq:DeltaE_bound2}
		&=: \sum_{\nu=0}^{\frac J2-1} C_{2c}(J,\nu)r^{J-2\nu}x^{J-\nu} + C_{2d}(J)x^\frac J2.
	\end{align}
	Now we estimate $R_K(2\ell;bZ,Z)$. From \eqref{eq:LZ2.3_RK}, we have the trivial bound
	\begin{equation}\label{eq:RK_triv}
		|R_K(2\ell;bZ,Z)| \le \frac{1}{(K-1)!} \int_0^1 \vb{\left[\left(\pd{}{w}\right)^{2\ell+2K} \sum_{r=1}^\infty (-1)^re^{-r^2Zt-rw}\right]_{w=bZ}} dt.
	\end{equation}
	Applying \Cref{lem:LZ2.4} with $z=-e^w$ and $h(r)=E_Z(r)$, we get
	\[
		\sum_{r=1}^\infty (-1)^re^{-r^2Z-rw} = -1 + \sum_{j=0}^{J-1} \frac{(-1)^je^{-jw}\De^j(E_Z(0))}{\left(1+e^{-w}\right)^{j+1}} + \frac{e^{-Jw}}{\left(1+e^{-w}\right)^J}\sum_{r=1}^\infty (-1)^re^{-rw}\De^J(E_Z(r)).
	\]
	Plugging this into \eqref{eq:RK_triv} with $Z\mapsto Zt$ and multiplying both sides by $e^{bZ}$, we get a bound
	\begin{multline}\label{eq:RK_2pieces}
		\vb{e^{bZ}R_K(2\ell;bZ,Z)} \le \frac{1}{(K-1)!}\left(\sum_{j=0}^{J-1} \int_0^1 \vb{e^{bZ}\left[\left(\pd{}{w}\right)^{2\ell+2K}\frac{(-1)^je^{-jw}\De^jE_{Zt}(0)}{\left(1+e^{-w}\right)^{j+1}}\right]_{w=bZ}} dt\right.\\
		\left.{\vphantom{\sum_{j=0}^{J-1}}}+ \int_0^1 \vb{e^{bZ}\left[\left(\pd{}{w}\right)^{2\ell+2K} \left(\frac{e^{-Jw}}{\left(1+e^{-w}\right)^J}\sum_{r=1}^\infty (-1)^re^{-rw}\De^J(E_{Zt}(r))\right)\right]_{w=bZ}} dt\right).
	\end{multline}

	We now consider the first integral in \eqref{eq:RK_2pieces}. For $k\in\N$ and $j\in\N_0$, let $C_{2e}(k,j)$ and $C_{2f}(j,x_0)$ be constants such that for all $w\in\C$ with $\re(w)\ge0$, $|\im(w)|\le\re(w)$, and $Z=x+iy$ with $0\le x\le x_0$, $|y|\le x$, we have
	\begin{equation}\label{eq:C67}
		\vb{e^w\left(\pd{}{w}\right)^k \frac{(-1)^je^{-jw}}{\left(1+e^{-w}\right)^{j+1}}} \le C_{2e}(k,j),\qquad \vb{\De^j(E_Z(0))} \le C_{2f}(j,x_0).
	\end{equation}
	Then we have, for $Z = x+iy$ with $0\le x \le x_0$ and $\vb{y}\le x$, 
	\begin{equation}\label{eq:RKp1_bound}
		\sum_{j=0}^{J-1} \int_0^1 \vb{e^{bZ}\left[\left(\tfrac{\partial}{\partial w}\right)^{2\ell+2K} \tfrac{(-1)^je^{-jw}}{\left(1+e^{-w}\right)^{j+1}}\right]_{w=bZ}}\De^j(E_{Zt}(0)) dt
		 \le \sum_{j=0}^{J-1} C_{2e}(2\ell+2K,j)C_{2f}(j,x_0).
	\end{equation}
	For the second term in \eqref{eq:RK_2pieces}, we use the product rule, and bound it by
	\begin{equation}\label{eq:RK_PR}
		\sum_{j=0}^{2\ell+2K} \binom{2\ell+2K}{j}\vb{e^{bZ}\left[\left(\tfrac{\partial}{\partial w}\right)^{2\ell+2K-j} \tfrac{e^{-Jw}}{\left(1+e^{-w}\right)^J}\right]_{w=bZ}}
		\int_0^1 \left|\sum_{r=1}^\infty (-1)^rr^je^{-rbZ}\De^J(E_{Zt}(r))\right| dt. 
	\end{equation}
	Now we estimate the inner sum in \eqref{eq:RK_PR}. Clearly, we have (as $\re(Z)=x\ge0$)
	\begin{equation}\label{eq:RKp2_inner}
		\Bigg|\sum_{r=1}^\infty (-1)^r r^j e^{-rbZ} \Delta^J (E_{Zt}(r))\Bigg| \le \sum_{r=1}^\infty r^j \vb{\Delta^J (E_{Zt}(r))}.
	\end{equation}
	Let $\vr>\frac12$ (chosen below). We split the sum in \eqref{eq:RKp2_inner} into two parts and write
	\begin{equation}\label{eq:RKp2_decomp}
		\sum_{r=1}^\infty r^j \vb{\Delta^J (E_{Zt}(r))} = \sum_{r=0}^{\pflo{(xt)^{-\varrho}}} r^j \vb{\Delta^J (E_{Zt}(r))} + \sum_{r=\pflo{(xt)^{-\varrho}}+1}^\infty r^j \vb{\Delta^J (E_{Zt}(r))}.
	\end{equation}
	For the first sum in \eqref{eq:RKp2_decomp}, we use \eqref{eq:DeltaE_bound2} to bound it against
	\begin{equation}\label{eq:RKp2_rsmall}
		\sum_{\nu=0}^{\frac J2-1} C_{2c}(J,\nu)(xt)^{J-\nu-\vr(J+j-2\nu+1)}+C_{2d}(J)\left((xt)^{\frac J2-\vr(j+1)}+\d_{j=0}(xt)^\frac J2\right).
	\end{equation}
	We choose $\vr$ such that all the exponents of $xt$ in \eqref{eq:RKp2_rsmall} are non-negative. If $J$ is sufficiently large (say $J\ge2\ell+2K+4$), then we can make this choice. Then \eqref{eq:RKp2_rsmall} is bounded by a constant independent of $x$ and $t$. For the second sum in \eqref{eq:RKp2_decomp}, we use the bound in \eqref{eq:Cnew02_def2} and obtain
	\begin{equation}\label{eq:RKp2_rlargeC2}
		\sum_{r=\pflo{(xt)^{-\varrho}}+1}^\infty r^j \vb{\Delta^J (E_{Zt}(r))} \le C_{2a}(J,x_0) \sum_{r=\pflo{(xt)^{-\varrho}}+1}^\infty r^j e^{-r^2xt}.
	\end{equation}
	Now assume that $(xt)^{-1}\ge(\frac j2)^{\frac12(\vr-\frac12)^{-1}}$. Note that $u^je^{-u^2xt}$ is decreasing for $u>(xt)^{-\vr}$. Thus use the integral comparison criterion, to obtain
	\begin{equation}\label{eq:Cnew08_def}
		\sum_{r=\pflo{(xt)^{-\vr}}+1}^\infty r^je^{-r^2xt} \le \int_{(xt)^{-\vr}}^\infty u^je^{-xtu^2} du \le C_{2g}(j,\vr),
	\end{equation}
	for a constant $C_{2g}(j,\vr)$ independent of $xt$. Plugging \eqref{eq:Cnew08_def} into \eqref{eq:RKp2_rlargeC2} yields
	\begin{equation}\label{eq:RKp2_rlarge}
		\sum_{r=\pflo{(xt)^{-\varrho}}+1}^\infty r^j \vb{\Delta^J (E_{Zt}(r))} \le C_{2a}(J,x_0) C_{2g}(j,\varrho)
	\end{equation}
	for $x_0\le(\frac j2)^{-\frac12(\vr-\frac12)^{-1}}$. Inserting \eqref{eq:RKp2_rsmall} and \eqref{eq:RKp2_rlarge} into \eqref{eq:RKp2_decomp}, we get
	\begin{multline*}
		\sum_{r=0}^\infty r^j \vb{\Delta^J (E_{Zt}(r))} \le \sum_{\nu=0}^{\frac J2-1} C_{2c}(J,\nu ) (x_0t)^{J-\nu  - \varrho(J+j-2\nu +1)}\\
		+ C_{2d}(J) \rb{(x_0t)^{\frac J2 - \varrho(j+1)} + \delta_{j=0} (x_0t)^{\frac J2}} + C_{2a}(J,x_0) C_{2g}(j,\varrho).
	\end{multline*}
	Integrating over $t$, we obtain
	\begin{equation}\label{eq:Cnew09_def}
		\int_0^1 \sum_{r=0}^\infty r^j \vb{\Delta^J ( E_{Zt}(r))} dt \le C_{2h}(J,j,x_0,\varrho),
	\end{equation}
	for $x_0\le(\frac j2)^{-\frac12(\vr-\frac12)^{-1}}$, where
	\begin{multline*}
		C_{2h}(J,j,x_0,\varrho) := \sum_{\nu=0}^{\frac J2-1} \frac{C_{2c}(J,\nu)}{J-\nu+1-\varrho(J+j-2\nu+1)} x_0^{J-\nu-\varrho(J+j-2\nu+1)}\\
	+ \frac{C_{2d}(J)}{\frac J2 + 1-\varrho(j+1)} x_0^{\frac J2 - \varrho(j+1)} +\frac{\delta_{j=0} C_{2d}(J)}{\frac J2 + 1} x_0^{\frac J2} + C_{2a}(J,x_0) C_{2g}(j,\varrho).
	\end{multline*}

	For $k\in\N_0$, let $C_{2i}(J,k)$ be constants such that for $w\in\C$ with $\re(w)\ge0$ and $|\im(w)|\le\re(w)$
	\begin{equation}\label{eq:Cnew10_def}
		\vb{e^w\left(\frac{\del}{\del w}\right)^k {\frac{e^{-Jw}}{\left(1+e^{-w}\right)^J}}} \le C_{2i}(J,k).
	\end{equation}
	Using \eqref{eq:RK_PR}, \eqref{eq:Cnew09_def}, and \eqref{eq:Cnew10_def}, the second integral in \eqref{eq:RK_2pieces} is bounded by 
	\begin{equation}\label{eq:RKp2_bound}
		\sum_{j=0}^{2\ell+2K} \binom{2\ell+2K}j C_{2i}(J, 2\ell+2K-j) C_{2h}(J,j,x_0, \varrho)
	\end{equation}
	for $x_0\le(\ell+K)^{-\frac12(\vr-\frac12)^{-1}}$ (this assumption is needed for \eqref{eq:Cnew09_def} to hold). Combining the error estimates \eqref{eq:RKp1_bound} and \eqref{eq:RKp2_bound}, and plugging into \eqref{eq:RK_2pieces}, we obtain that
	\begin{multline*}
		\vb{e^{bZ} R_K(2\ell;bZ,Z)} \le \frac{1}{(K-1)!} \left(\sum_{j=0}^{J-1} C_{2e}(2\ell+2K,j) C_{2f}(j,x_0)\right.\\
		+ \left.\sum_{j=0}^{2\ell+2K} \binom{2\ell+2K}{j} C_{2i}(J,2\ell+2K-j) C_{2h}(J,j,x_0,\varrho)\right),
	\end{multline*}
	which holds for $Z=x+iy$, $0\le x\le x_0\le(\ell+K)^{-\frac12(\vr-\frac12)^{-1}}$, $|y|\le x$. Plugging back into \eqref{eq:RK_applied}, we obtain, for $\b_n\le x_0\le\frac23(\ell+K)^{-\frac12(\vr-\frac12)^{-1}}$,
	\begin{multline}\label{eq:bound_CCCC}
		\vb{\sum_{r=1}^\infty (-1)^{r+1}r^{2\ell}e^{-\left(\frac{3r^2}{2}+\left(m+\frac12\right)r\right)Z} - \sum_{k=0}^{K-1} \frac{\left(-\frac{3Z}{2}\right)^k}{k!}\left[\left(\frac{\del}{\del w}\right)^{2\ell+2k} {\frac{1}{1+e^w}}\right]_{w=\left(m+\frac12\right)Z}}\\
		\le \frac{3^K\b_n^Ke^{-\left(m+\frac12\right)\b_n}}{2^\frac K2(K-1)!}\left(\sum_{j=0}^{J-1} C_{2e}(2\ell+2K,j)C_{2f}\left(j,\frac{3x_0}{2}\right)\right.\\
		+ \left.\sum_{j=0}^{2\ell+2K} \binom{2\ell+2K}{j}C_{2i}(J,2\ell+2K-j)C_{2h}\rb{J,j,\frac{3x_0}{2},\vr}\right).
	\end{multline}
	Multiplying both sides of \eqref{eq:bound_CCCC} by $\frac{2Z^{2\ell}}{(2\ell)!}$ we obtain, again using that $|Z|\le\sqrt{2}x$
	\begin{align*}
		&\vb{\Mc_{\ell,m}(Z)-\frac{2}{(2\ell)!}\sum_{k=0}^{K-1} \frac{\left(-\frac32\right)^kZ^{2\ell+k}}{k!}\left[\left(\pd{}{w}\right)^{2\ell+2k} \frac{1}{1+e^w}\right]_{w=\left(m+\frac12\right)Z}}\\
		&\hspace{1cm}\le \frac{2^{\ell-\frac K2+1}3^K}{(2\ell)!(K-1)!}\b_n^{2\ell+K} e^{-\left(m+\frac12\right)\b_n}\left(\sum_{j=0}^{J-1} C_{2e}(2\ell+2K,j)C_{2f}\left(j,\frac{3x_0}{2}\right)\right.\\
		&\hspace{6cm}\left.+ \sum_{j=0}^{2\ell+2K} \binom{2\ell+2K}{j}C_{2i}(J,2\ell+2K-j) C_{2h}\left(J,j,\frac{3x_0}{2},\vr\right)\right)\\
		&\hspace{1cm}=: C_2(\ell,K,J,x_0,\vr)\b_n^{2\ell+K}e^{-\left(m+\frac12\right)\b_n}. \qedhere
	\end{align*}
\end{proof}

\subsection{Back to the Circle Method}
{We now combine the above results and use the Circle Method.}
By \eqref{eq:Cauchy_int}, we have a splitting 
\[
	N(m,n) - N(m+1,n) = \mf M + \mf m,
\]
where $\mf M$ and $\mf m$ are given in \eqref{eq:maj_arc_int} and \eqref{eq:min_arc_int}, respectively. From \Cref{prp:min_arc_bound} we know that
\[
	\vb{\mf m} \le 13 e^{\frac {3\mc C_n}2}.
\]
Meanwhile, by \eqref{eq:M1_def}, we have a splitting
\[
	\mf M = \mc M^{[1]} + \mc E^{[1]},
\]
with {\hlb (see \eqref{eq:E1_display})}
\[
	\vb{\mc E^{[1]}} \le e^{\mc C_n}. 
\]

A straightforward calculation then shows the following.

\begin{lemma}\label{lem:magic32}
	Let $n\in\N$. Define $\b_n$ and $\Cc_n$ as in \eqref{eq:bL_def} and let $Z=\b_n+iy$. For $\nu\in\N_0$, we have
	\[
		\int_{\vb{y}\le \beta_n} \vb{y}^\nu  \vb{e^{\left(n-\frac{1}{24}\right)Z+\frac{\pi^2}{6Z}}} dy \le \frac{12^{\frac{\nu +1}{2}}}{\pi^{\nu +1}} \Gamma\rb{\frac{\nu +1}{2}} \beta_n^{\frac 32 (\nu +1)} e^{2\mathcal{C}_n}.
	\]
\end{lemma}

Next we give a decomposition of $\Mc^{[1]}$ which follows directly from \eqref{eq:M1_def} and \Cref{lem:ELmz_bound,lem:magic32}.

\begin{lemma}\label{lem:M1_split}
For $L\ge 2$, we have a decomposition
\[
	\mc M^{[1]} = \sum_{\ell = 1}^{L-1} \mc M_\ell^{[2]} + \mc E_L^{[2]},
\]
where
\begin{align}\label{eq:M2_def}
	\mc M_\ell^{[2]} &:= \;\frac{1}{(2\pi)^{\frac 32}} \int_{\vb{y}\le \beta_n} \sqrt{Z} e^{\left(n-\frac{1}{24}\right)Z +\frac{\pi^2}{6Z}} \mc M_{\ell,m}(Z) dy,\\
	\mc E_L^{[2]} &:= \;\frac{1}{(2\pi)^{\frac 32}} \int_{\vb{y}\le \beta_n} \sqrt{Z} e^{\left(n-\frac{1}{24}\right)Z +\frac{\pi^2}{6Z}} \mc E_{L,m}(Z) dy. \label{eq:E2_def}
\end{align}
Moreover, the error term $\mc E_L^{[2]}$ satisfies,  for $\beta_n \leq x_0$
\begin{equation}
	\vb{\mc E_L^{[2]}} \le \frac{\sqrt{3}}{2^\frac14\pi^2}C_1(L,x_0)\b_n^{L+\frac32}e^{2\mc C_n-\left(m+\frac12\right)\b_n}
	=: \mc D^{[2]}_L(x_0) \beta_n^{L+\frac 32} e^{2\mathcal{C}_n-\left(m+\frac 12\right)\beta_n}.\label{eq:E2_display}
\end{equation}

\end{lemma}

From \Cref{prop:Cnew11_def} and \Cref{lem:magic32} we obtain {the following lemma}. 

\begin{lemma}\label{lem:M2_split}
For $\ell,K\in\N$, we have a decomposition
\ba
	\mc M_\ell^{[2]} = \sum_{k=0}^{K-1} \mc M_{\ell,k}^{[3]} + \mc E_{\ell,K}^{[3]},
\ea
where
\begin{equation*}
	\mc M_{\ell,k}^{[3]} := \frac{2 \left(-\frac 32\right)^k}{(2\pi)^{\frac 32}(2\ell)! k!} \int_{\vb{y}\le \beta_n} Z^{2\ell+k+\frac 12} e^{\rb{n-\frac 1{24}}Z+\frac{\pi^2}{6Z}} \sqb{\rb{\frac{\partial}{\partial w }}^{2\ell+2k} {\frac{1}{1+e^{w}}}}_{w = \left(m+\frac 12\right)Z} dy.
\end{equation*}
The error term $\mc E_{\ell,K}^{[3]}$ satisfies, for $\b_n\le x_0\le\frac23(\ell+K)^{-\frac12(\vr-\frac12)^{-1}}$
{\begin{equation*}
\scalebox{0.95}{$\displaystyle\vb{\mc E_{\ell,K}^{[3]}} \le \frac{\sqrt{3}}{2^\frac14\pi^2}C_2(\ell,K,J,x_0,\vr) \b_n^{2\ell+K+2}e^{2\mc C_n-\left(m+\frac12\right)\b_n} =: \mc D^{[3]}_{\ell,K}(J,x_0,\varrho) \beta_n^{2\ell+K+2} e^{2\Cc_n - \left(m+\frac 12\right)\beta_n}.$} 
\end{equation*}}
\end{lemma}

Next we give a decomposition of $\mc M_{\ell,k}^{[3]}$.

\begin{lemma}\label{lem:M3_split}
For $\ell,R\in\N$, $k\in\N_0$, we have a decomposition
\[
	\mc M_{\ell,k}^{[3]} = \sum_{r=0}^{R-1} \mc M_{\ell,k,r}^{[4]} + \mc E_{\ell,k,R}^{[4]},
\]
where
\begin{equation}\label{eq:cMlkv2_def}
	\Mc_{\ell,k,r}^{[4]} \hspace{-1mm}:=\hspace{-1mm} \tfrac{2\left(-\frac32\right)^k\left(m+\frac12\right)^r}{\sqrt{2\pi}(2\ell)!k!r!}\hspace{-1mm} \left[\left(\tfrac{\partial}{\partial w}\right)^{2\ell+2k+r}\hspace{-1mm} \tfrac{1}{1+e^w}\right]_{w=\left(m+\frac12\right)\b_n}\hspace{-1mm}
	 \tfrac{1}{2\pi}\int_{|y|\le\b_n}\hspace{-2mm} Z^{2\ell+k+\frac12}e^{\left(n-\frac{1}{24}\right)Z+\frac{\pi^2}{6Z}}(iy)^r dy.
\end{equation}
The error term $\mc E_{\ell,k,R}^{[4]}$ satisfies
\begin{align}\nonumber
	\vb{\mc E_{\ell,k,R}^{[4]}} &\le \frac{2^{\ell-\frac k2+R+\frac34}3^{k+\frac{R+1}{2}} \Ga\left(\frac{R+1}{2}\right)\left(m+\frac12\right)^R}{\pi^{R+\frac52}(2\ell)!k!(R-1)!} C_{2e}(2\ell+2k+R,0)\b_n^{2\ell+k+\frac{3R}{2}+2}e^{2\Cc_n-\left(m+\frac12\right)\b_n}.\\
	&=: \mc D_{\ell,k,R}^{[4]} {\textstyle \left(m+\frac 12\right)^R} \beta_n^{2\ell+k+\frac{3R}{2}+2} e^{2\mathcal{C}_n - \left(m+\frac 12\right)\beta_n}.\label{eq:E4_display}
\end{align}
\end{lemma}
\begin{proof}
For $j\in\N_0$, we have Taylor series expansion in $Z$ at $Z=\b_n$
\begin{multline}\label{eq:pdu_Taylor}
	\left[\left(\pd{}{w}\right)^j \frac{1}{1+e^w}\right]_{w=\left(m+\frac12\right)Z} = \sum_{r=0}^{R-1} \left[\left(\pd{}{w}\right)^{j+r} \frac{1}{1+e^w}\right]_{w=\left(m+\frac12\right)\b_n} \frac{\left(m+\frac12\right)^r(iy)^r}{r!}\\
	+ \frac{i^R\left(m+\frac12\right)^R}{(R-1)!}\int_0^y (y-u)^{R-1}\left[\left(\pd{}{w}\right)^{j+R} \frac{1}{1+e^w}\right]_{w=\left(m+\frac12\right)(\b_n+iu)} du.
\end{multline}
It follows by \eqref{eq:C67} that the error term in \eqref{eq:pdu_Taylor} is bounded by
\begin{equation}\label{eq:Cnew06_Taylor}
	\frac{\left(m+\frac12\right)^R |y|^R}{(R-1)!}C_{2e}(j+R,0)e^{-\left(m+\frac12\right)\b_n}.
\end{equation}
The lemma then follows from \eqref{eq:Cnew06_Taylor} and \Cref{lem:magic32}.
\end{proof}

{We need yet another decomposition.} 

\begin{lemma}\label{lem:M4_split}
For $\ell,\Lambda\in\N$ and $k,r\in\N_0$, we have a decomposition 
\[
	\Mc_{\ell,k,r}^{[4]} = \sum_{l=\pcei{\frac r2}}^{\La-1} \Mc_{\ell,k,r,l}^{[5]} + \Ec_{\ell,k,r,\La}^{[5]},
\]
where
\begin{multline}\label{eq:M5_def}
	\mc M_{\ell,k,r,l}^{[5]} := \frac{(-1)^l \left(-\frac 32\right)^k 3^{l+\frac 12}}{\pi^{2l+2} (2\ell)!k!r!} a_l^{[r ]} \left(-2\ell-k-\tfrac 32\right) \left(m+\tfrac 12\right)^r   \\
	\times \left[\left(\frac{\partial}{\partial w}\right)^{2\ell+2k+r}{\frac{1}{1+e^w}}\right]_{w = \left(m+\frac 12\right)\beta_n} \beta_n^{2\ell+k+r +l+2} e^{2\mathcal{C}_n},
\end{multline}
and the error term $\mc E_{\ell,k,r,\La}^{[5]}$ satisfies for $\b_n\le x_0$ and some constant $\Dc_{\ell,k,r,\La}^{[5]}(x_0)$
\begin{equation}\label{eq:E5_display}
	\vb{\Ec_{\ell,k,r,\La}^{[5]}} \le \mc D_{\ell,k,r ,\Lambda}^{[5]}(x_0) { \left(m+\tfrac 12\right)}^r  \beta_n^{2\ell+k+r +\Lambda+2} e^{2\mathcal{C}_n-\left(m+\frac 12\right)\beta_n}.
\end{equation}

\end{lemma}

\begin{proof}
	We make the change of variable $Z\mapsto\b_nZ$ to obtain that the integral in \eqref{eq:cMlkv2_def} (including the previous $\frac{1}{2\pi}$) equals
	\begin{equation}\label{eq:cMlkv2_decomp}
		\b_n^{2\ell+k+r+\frac32}\sum_{j=0}^r \binom rj(-1)^{r+j}\mc I_{-2\ell-k-j-\frac32}(2\Cc_n).
	\end{equation}
	Applying \Cref{lem:Bessel_segment}, we can approximate the sum in \eqref{eq:cMlkv2_decomp} with $I$-Bessel functions. To be more precise, we have, for $\beta_n \leq x_0$
	\begin{multline}\label{eq:Cnew12_def}
		\b_n^{2\ell+k+r+\frac32}\vb{\sum_{j=0}^r \binom rj(-1)^{r+j}\left(\Ic_{-2\ell-k-j-\frac32}(2\Cc_n)-I_{-2\ell-k-j-\frac32}(2\Cc_n)\right)}\\
		\le \frac{\b_n^{2\ell+k+r+\frac32}}{\pi}\sum_{j=0}^r \binom rj2^{\ell+\frac k2+\frac j2+\frac14}\left(\frac{\Ga\left(2\ell+k+j+\frac32\right)}{\Cc_n^{2\ell+k+j+\frac32}}+2e^\frac{3\Cc_n}{2}\right) \le C_{3a}(\ell,k,r,x_0)e^\frac{3\Cc_n}{2}
	\end{multline}
	for some constant $C_{3a}(\ell,k,r,x_0)$. Now we apply \eqref{eq:Besselv_asymp} (with $M=0$) and obtain for $\La\in\N$ 
	\begin{align}\nonumber
		&\b_n^{2\ell+k+r+\frac32} \left|I_{-2\ell-k-\frac32}^{[r]}(2\Cc_n) - \frac{e^{2\Cc_n}}{2\sqrt{\pi\Cc_n}} \sum_{l=\pcei{\frac r2}}^{\Lambda-1} (-1)^l a_l^{[r]}\left(-2\ell-k-\tfrac32\right)(2\Cc_n)^{-l}\right|\\
		&\hspace{1.5cm}\le \b_n^{2\ell+k+r+\frac32} \left(\frac{e^{2\Cc_n}}{2\sqrt{\pi\Cc_n}} \left|\d_\Lambda^{[r]}\left(-2\ell-k-\tfrac32,2\Cc_n\right)\right| + \frac{e^{-2\Cc_n}}{2\sqrt{\pi\Cc_n}}\left|\g_0^{[r]}\left(-2\ell-k-\tfrac32,2\Cc_n\right)\right|\right)\nonumber\\
		&\hspace{1.5cm} \le C_{3b}(\ell,k,r,\Lambda,x_0)\b_n^{2\ell+k+r+\Lambda+2}e^{2\Cc_n}\label{eq:Cnew13_def}
	\end{align}
	for $\b_n\le x_0$ and some constant $C_{3b}(\ell,k,r,\La,x_0)$. Plugging this back into \eqref{eq:cMlkv2_decomp}, we conclude that
	\begin{multline}\label{eq:C1213_bound}
		\vb{\frac{1}{2\pi}\int_{|y|\le\b_n} Z^{2\ell+k+\frac12} e^{\left(n-\frac{1}{24}\right)Z+\frac{\pi^2}{6Z}}(iy)^r dy - \sum_{l=\pcei{\frac r2}}^{\La-1} \frac{(-1)^l3^{l+\frac12}}{\sqrt{2}\pi^{2l+\frac32}}a_l^{[r]}\left(-2\ell-k-\frac32\right)\b_n^{2\ell+k+r+l+2}e^{2\Cc_n}}\\
		\le C_{3b}(\ell,k,r,\La,x_0)\b_n^{2\ell+k+r+\La+2}e^{2\Cc_n} + C_{3a}(\ell,k,r,x_0)e^\frac{3\Cc_n}{2}
	\end{multline}
	for $\beta_n\le x_0$, by \eqref{eq:Cnew12_def} and \eqref{eq:Cnew13_def}.  { It then follows from \eqref{eq:C1213_bound} that 
	\begin{align*}\nonumber
		\vb{\mc E_{\ell,k,r,\Lambda}^{[5]}} &\le \frac{2\left(\frac32\right)^k\left(m+\frac12\right)^r}{\sqrt{2\pi}(2\ell)!k!r!} C_{3c}(2\ell+2k+r)\\
		&\times \left(C_{3b}(\ell,k,r,\La,x_0)\b_n^{2\ell+k+r+\La+2}e^{2\Cc_n-\left(m+\frac12\right)\b_n} + C_{3a}(\ell,k,r,x_0)e^{\frac32\Cc_n-\left(m+\frac12\right)\b_n}\right)
	\end{align*}
	for $\b_n\le x_0$, where} $C_{3c}(j)$ is a constant such that, for $w\in\R_0^+$, 
	\begin{equation}\label{eq:Cnew14_def}
		\vb{e^w \left(\frac{ \partial}{\partial w}\right)^j {\frac{1}{1+e^w}}} \le C_{3c}(j).
	\end{equation}
	Now let $\mc D_{\ell,k,r ,\Lambda}^{[5]}(x_0)$ be constants such that
	\begin{align*}
		&\frac{2 \left(\frac 32\right)^k \left(m+\frac 12\right)^r }{\sqrt{2\pi} (2\ell)!k!r !} C_{3c}(2\ell+2k+r )\\
		&\hspace{15mm}\times \left(C_{3b}(\ell,k,r,\Lambda,x_0) + C_{3a}(\ell,k,r ,x_0) e^{-\frac{\Cc_n}{2}} \beta_n^{-(2\ell+k+r+\Lambda+2)}\right) \beta_n^{2\ell+k+r+\Lambda+2} e^{2\Cc_n-\left(m+\frac 12\right)\beta_n}\\
		&\hspace{75mm} \le \mc D_{\ell,k,r ,\Lambda}^{[5]}(x_0) {\textstyle \left(m+\frac 12\right)}^r  \beta_n^{2\ell+k+r +\Lambda+2} e^{2\mathcal{C}_n-\left(m+\frac 12\right)\beta_n}
	\end{align*}
	for $\b_n\le x_0$. It is straightforward to see that $\Dc_{\ell,k,r,\La}^{[5]}(x_0)$ exists.  The lemma now follows.
\end{proof}

\subsection{An extra decomposition}
{We conclude this section with one final decomposition we require in the proof of our main result.}
In this subsection, assume that $m\ll\b_n^{-1}$. For $j\in\N_0$, consider
\begin{equation}\label{fjdefn}
	f_j(u) := e^u \left(\frac{\partial}{\partial u}\right)^j {\frac{1}{1+e^u}}. 
\end{equation}
By Taylor's Theorem, for $\g\ge0$, there is a constant $C_4(j,N,\g)$ such that, for $0\le u\le\g$,
\begin{equation}\label{eq:Cnew15_def}
	\left|f_j(u)-\sum_{\nu=0}^{N-1} f_j^{(\nu)}(0)\frac{u^\nu}{\nu!}\right| \le C_4(j,N,\g)u^N.
\end{equation}
Using the Taylor expansion above, we immediately obtain the following lemma.
\begin{lemma}\label{lem:M5_split}
Let $\ell,N\in\N$, $k,r,l\in\N_0$ Then we have a decomposition
\[
	\mc M_{\ell,k,r,l}^{[5]} = \sum_{\nu=0}^{N-1} \mc M_{\ell,k,r,l,\nu}^{[6]} + \mc E_{\ell,k,r ,l,N}^{[6]},
\]
where
\begin{align}\nonumber
	\Mc_{\ell,k,r,l,\nu}^{[6]} &:= \frac{(-1)^l\left(-\frac32\right)^k3^{l+\frac12}}{\pi^{2l+2}(2\ell)!k!r!\nu!} a_l^{[r]}\left(-2\ell-k-\tfrac32\right)\left(m+\tfrac12\right)^{r+\nu}f_{2\ell+2k+r}^{(\nu)}(0) \b_n^{2\ell+k+r+l+\nu+2}\\
	&\hspace{10.5cm}\times e^{2\Cc_n-\left(m+\frac12\right)\b_n},\label{eq:M6_def}\\
	\Ec_{\ell,k,r,l,N}^{[6]} &:= \frac{\left(\tfrac32\right)^k3^{l+\tfrac12}}{\pi^{2l+2}(2\ell)!k!r!} \vb{a_l^{[r]}\left(-2\ell-k-\tfrac32\right)}\left(m+\tfrac12\right)^{r+N}\nonumber\\
	&\hspace{-1cm}\times \left(f_{2\ell+2k+r}\left(\left(m+\tfrac12\right)\b_n\right)-\sum_{\nu=0}^{N-1} f_{2\ell+2k+r}^{(\nu)}(0)\frac{\left(m+\frac12\right)^\nu\b_n^\nu}{\nu!}\right) \b_n^{2\ell+k+r+l+N+2}e^{2\Cc_n-\left(m+\frac12\right)\b_n}.\nonumber
\end{align}
For $\gamma>0$ and $(m+\frac 12)\beta_n \le \gamma$, the error term $\mc E_{\ell,k,r,l,N}^{[6]}$ satisfies
\begin{align}\nonumber
	\vb{\Ec_{\ell,k,r,l,N}^{[6]}} &\le \frac{\left(\tfrac32\right)^k3^{l+\frac12}}{\pi^{2l+2}(2\ell)!k!r!} \vb{a_l^{[r]}\left(-2\ell-k-\tfrac32\right)}\left(m+\tfrac12\right)^{r+N}C_4(2\ell+2k+r,N,\g)\\
	&\hspace{3cm} \times \b_n^{2\ell+k+r+l+N+2}e^{2\Cc_n-\left(m+\frac12\right)\b_n}\nonumber\\
	&=: \Dc_{\ell,k,r,l,N}^{[6]}(\g) \b_n^{2\ell+k+r+l+N+2} e^{2\Cc_n-\left(m+\frac12\right)\b_n}\label{eq:E6_display}.
\end{align}
\end{lemma}

\section{Numerical computations}\label{section:num_com}

In this section, we prove the following theorem which establishes our main result for small $m$. Recall the definitions of $\b_n$ and $\La_n$ from \eqref{eq:bL_def}.

\begin{theorem}\label{thm:m_small}
	Let $n\ge925276$. Then we have, for {$m\le 3^{-1}\b_n^{-\frac32}-\frac 12$}, 
\begin{equation}
		N(m,n) - N(m+1,n) \ge 0.\label{rankdiff}
\end{equation}
\end{theorem}

{The strategy is to apply the decompositions above, and choose parameters such that the sizes of the error terms are quite small. The parameters we choose also depend on the size of $m$. Precisely, for the proof of the theorem, we pick $(x_0,L,K,R,\Lambda,N,J) = (\frac{1}{750},5,4,1,2,2,50)$ for $m+\frac 12 \le \beta_n^{-1}$, and $(x_0,L,K,R,\Lambda,J) = (\frac{1}{750},5,4,4,2,50)$ for $\frac{1}{\beta_n} \le m+\frac 12 \le 3^{-1}\b_n^{-\frac32}$. Once these parameters are chosen, we need to compute the following constants to obtain the error bounds:} 
\begin{itemize}[leftmargin=*]
	\item $C_{2a}$: We need $C_{2a}(50,\frac{1}{500})$.
	
	\item $C_{2b}$: We need $C_{2b}(50,\nu)$ for $0\le \nu\le 24$. 
	
	\item $C_{2e}$: We need $C_{2e}(k,j)$ for $k \in \cb{10,12,14,16}$, $0\le j \le 49$, as well as $C_{2e}(k,0)$ for $2\le k \le 18$. 
	
	\item $C_{2f}$: We need $C_{2f}(j,\frac{1}{500})$ for $0\le j \le 49$.
	
	\item $C_{2g}$: We need $C_{2g}(j,\vr)$ for $0\le j\le16$. Note that at this stage the parameter $\vr$ is not specified. In principle, we want to choose $\vr$ such that the overall error is as small as possible. This is however quite tedious. Therefore we instead pick the parameter $\vr$ such that the overall error is quite small. Hiding the trial-and-error steps, we are interested in $C_{2g}(j,\vr)$ for $0\le j\le16$, and $\vr\in\{{ 0.66021,0.66950,0.67773,0.68510}\}$.
	
	\item $C_{2i}$: We need $C_{2i}(50,k)$ for $0\le k \le 16$.
	
	\item $C_{3a}$: We need $C_{3a}(\ell,k,r,\frac{1}{750})$ for $1\le \ell \le 4$, $0\le k \le 3$, $0 \le r \le 3$.
	
	\item $C_{3b}$: We need $C_{3b}(\ell,k,r,2,\frac{1}{750})$ for $1\le \ell \le 4$, $0\le k \le 3$, $0\le r \le 3$.
	
	\item $C_{3c}$: We need $C_{3c}(j)$ for $2\le j \le 17$.
	
	\item $C_4$: We need to compute $C_4(j,2,1)$ for $2\le j \le 14$.
\end{itemize}

The constants $C_1, C_2, C_{2c},C_{2d}$, and $C_{2h}$ are defined explicitly, possibly in terms of other constants. The constants $C_{2b},C_{2e},C_{2f},C_{2g},C_{2i},C_{3a},C_{3b},C_{3c},$ and $C_4$ are found using one-dimensional optimization. We now give an outline on how to compute these.
\begin{enumerate}[label={(\arabic*)},wide,labelwidth=!,labelindent=0pt]
	\item For $C_{2b}$, we have to compute $C_{2b}(50,\nu)$ for $0\le\nu\le24$. As the defining function \eqref{eq:Cnew03_def} decays as $\re(w)\to\infty$, it suffices to consider the function over a (sufficiently large) compact region in $Z$. By the maximum modulus principle, the maximal modulus of the function occurs at the boundary of the region. So we may find an optimal value for $C_{2b}(50,\nu)$ using a one-dimensional plot along the lines $|\im(w)|=\re(w)$. The same argument applies for $C_{2e}$ and $C_{2i}$ (see \eqref{eq:C67} and \eqref{eq:Cnew10_def}) if the parameters are fixed. For $C_{2f}$ (see \eqref{eq:C67}), the region under consideration is by definition bounded, so we do not need to consider limits.
	
	\item For $C_{2g}$, $C_{3a}$, $C_{3b}$, $C_{3c}$, and $C_4$, we can find their optimal values (if the parameters are fixed) using real-valued functions (see \eqref{eq:Cnew08_def}, \eqref{eq:Cnew12_def}, \eqref{eq:Cnew13_def}, \eqref{eq:Cnew14_def}, and \eqref{eq:Cnew15_def}, respectively).
\end{enumerate}

It remains to compute $C_{2a}(50,\frac{1}{500})$, which is computationally more difficult. We recall the definition in \eqref{eq:Cnew02_def}. We show that we may take
\begin{equation}\label{eq:Cnew02_numeric}
	C_{2a} \left(50,\tfrac{1}{500}\right) = \vb{\sum_{j=0}^{50} \binom{50}{j}(-1)^je^{-\left(2\cdot563j+j^2\right)\frac{1+i}{500}}} \approx 30.02166509.
\end{equation}
We start with some observations.
\begin{enumerate}[label={(\arabic*)},wide,labelwidth=!,labelindent=0pt]
	\item If $r$ is fixed, then
	\begin{equation*}
		\sum_{j=0}^{50} \binom {50}j (-1)^j e^{-\left(2rj+j^2\right)Z}
	\end{equation*}
	is  holomorphic in $Z$. By the maximum modulus principle, the left-hand side of \eqref{eq:Cnew02_def} obtains its maximum on the boundary of $0\le \vb{y} \le x \le \frac{1}{500}$. That is, \eqref{eq:Cnew02_def} is maximal if $y=\pm x$, or if $0\le \vb{y} \le x = \frac{1}{500}$. As \eqref{eq:Cnew02_def} is invariant under complex conjugation, we may assume that $y\ge 0$.
	
	\item First assume that $rx\ge\frac32$. Then $|e^{-2rZ}|\le e^{-3}$, and we have the trivial bound
	\ba
		\vb{\sum_{j=0}^{50} \binom{50}j (-1)^j e^{-\left(2rj+j^2\right)Z}} \le \sum_{j=0}^{50} \binom{50}j e^{-3j} = \rb{1+e^{-3}}^{50} \approx 11.35170075.
	\ea
	So the function satisfies the bound \eqref{eq:Cnew02_numeric} if $rx\ge\frac32$. It suffices to consider the case where $0\le rx\le\frac32$.
\end{enumerate}

First consider the easier case $0\le y\le x=\frac{1}{500}$ from observation (1). By observation (2) we only have to consider the case $r\le\frac32\cdot500=750$. As $r\in\N_0$, we can verify that our constant works in this case, with $x=\frac{1}{500}$ and $0\le r\le750$ fixed, and let $0\le y\le\frac{1}{500}$ vary.

In the trickier case $0 \le y = x \le \frac{1}{500}$ from observation (1), \eqref{eq:Cnew02_def} becomes
\begin{equation}\label{eq:Cnew02_xy}
	\vb{\sum_{j=0}^{50} \binom {50}j (-1)^j e^{-\left(2rjx+j^2x\right)(1+i)}} \le C_{2a}(J,x_0).
\end{equation}
Let 
\ba
	f(\rho,x) := \sum_{j=0}^{50} \binom {50}j (-1)^j e^{-\left(2j\rho+j^2x\right)(1+i)}.
\ea
The function on the left-hand side of \eqref{eq:Cnew02_xy} is $|f(rx,x)|$. By observation (2) we only have to consider the compact region $0\le\rho\le\frac32$, $0\le x\le\frac{1}{500}$. We split $f(\rho,x)$ into real and imaginary parts
\ba
	f(\rho,x) = \sum_{j=0}^{50} \binom {50}j (-1)^j e^{-\left(2j\rho+j^2x\right)} \rb{\cos\rb{2j\rho+j^2x} - i\sin\rb{2j\rho+j^2x}},
\ea
and use this to obtain an expression of $\vb{f(\rho,x)}^2$
\ba
	\scalebox{0.9}{$\displaystyle \vb{f(\rho,x)}^2 = \rb{\sum_{j=0}^{50} \binom {50}j (-1)^j e^{-\left(2j\rho+j^2x\right)} \cos\rb{2j\rho+j^2x}}^2 + \rb{\sum_{j=0}^{50} \binom {50}j (-1)^j e^{-\left(2j\rho+j^2x\right)} \sin\rb{2j\rho+j^2x}}^2.$}
\ea
We consider the Taylor series expansion for $\vb{f(\rho,x)}$ at $x=0$ as a function of $x$
\ba
	\vb{f(\rho,x)} = \sum_{\ell=0}^{L-1} \frac{c_\ell(\rho)}{\ell!} x^\ell + R_L.
\ea
Here $c_\ell$ is a function of $\rho$, and for $x\ge0$, $R_L$ is bounded by
\ba
	\vb{R_L} \le \frac{x^L}{L!} \max_{0\le u\le x}\vb{\pd{^L}{u^L} \vb{f(\rho,u)}} \le \frac{x^L}{L!} \sum_{j=0}^{50} \binom{50}j j^{2L}.
\ea
Let $d_\ell := \max\cb{c_\ell(\rho) \;:\; 0\le\rho\le\frac 32}$. Then we have for $0\le \rho \le \frac 32$ and $x\ge 0$
\begin{equation}\label{eq:frx_ub}
	\vb{f(\rho,x)} \le \sum_{\ell=0}^{L-1} \frac{d_\ell}{\ell!} x^\ell + \frac{x^L}{L!} \sum_{j=0}^{50} \binom{50}j j^{2L}.
\end{equation}

We check that
\ba
	d_0&\le 28.68558614, & d_1 &\le 1220.781415, & d_2&\le 99722.52362, & d_3 &\le 12254165.90,\\
	d_4&\le 2040711473, & d_5 &\le 4.319895682\cdot 10^{11}, & d_6&\le 1.111629852\cdot 10^{14},\\
	d_7 &\le 3.373057489\cdot 10^{16}, & d_8&\le 1.194739620\cdot 10^{19}, & d_9 &\le 5.092346787\cdot 10^{21}.
\ea
Plugging into to \eqref{eq:frx_ub}, we compute that 
\ba
	\vb{f(\rho,x)} \le 29.70793207 \text{ for } 0\le \rho \le \tfrac 32, \; 0 \le x \le \tfrac{1}{6000}.
\ea
So it remains to check the region $\frac{1}{6000}\le x\le\frac{1}{500}$. Since we may assume from observation (2) that $0\le rx\le\frac32$, and $r\in\N_0$, we can verify that our constant works in this case by considering $0\le r\le9000$ fixed, and let $0\le x\le\frac{1}{500}$ vary.

\begin{remark}
	For our calculations, we always pick the optimal values for these named constants, computed using numerical methods. As usual, values used for error bounds are always rounded up, and values used for main terms are always rounded down. For the convenience to the reader, we also sometimes display the approximated values for the constants. These approximations are marked with the approximation sign $\approx$.
\end{remark}
{With these observations in mind we now prove the main result of this section.}
\begin{proof}[Proof of \Cref{thm:m_small}]
	First suppose that $m+\frac 12\le \frac1{\beta_n}$, i.e., we set $\gamma = 1$. We pick $x_0 = \frac1{750}$, and $L=5$, $K=4$, $R=1$, $\Lambda=2$, $N=2$, and $J=50$. Since the constants $C_j$ are now determined, we are able to compute the error bounds $\Dc^{[j]}$. From \eqref{eq:M6_def}, the main term is given by
	\[
		\Mc_{1,0,0,0,1}^{[6]} \ge 0.1082531754 \left(m+\tfrac12\right)\b_n^5e^{2\La_n-\left(m+\frac12\right)\b_n}. 
	\]
	Next we compute the error terms. We ``optimize'' the error $\Dc_{\ell,K}^{[3]}$ with respect to the parameter $\vr$
	\begin{align*}
		\Dc_{1,4}^{[3]}(x_0,J,{ 0.66021}) &\approx { 6.448427833}\cdot10^6, &\Dc_{2,4}^{[3]}(x_0,J,{ 0.66950}) &\approx {1.424375158}\cdot10^{10},\\
		\Dc_{3,4}^{[3]}(x_0,J,{0.67773}) &\approx { 1.394394143}\cdot10^{13}, &\Dc_{4,4}^{[3]}(x_0,J,{ 0.68510}) &\approx { 8.019576913}\cdot10^{15}.
	\end{align*}
	Using these bounds, we compute for $\beta_n \le x_0$ that the contribution from $\mc M_{\ell,k,r,l,\nu}^{[6]}$ for $(\ell,k,r,l,\nu)\ne (1,0,0,0,1)$ and all the error terms is bounded by
	\[
		\scalebox{0.98}{$\Dc_a := \left({ 0.02839882501}\b_n^5 + 0.01275733148\left(m+\frac12\right)\b_n^5 + 0.01145145086\left(m+\frac12\right)^2\b_n^6\right) e^{2\Cc_n-\left(m+\frac12\right)\b_n}.$}
	\]
	As $m\ge 0$ and $(m+\frac 12)\beta_n \le 1$, we conclude that	for $\b_n\le x_0$, and $m+\frac12\le\frac1{\b_n}$,
	\[
		\frac{\Dc_a}{\Mc_{1,0,0,0,1}^{[6]}} \le \frac{2\cdot{0.02839882501}+0.01275733148+0.01145145086}{0.1082531754} \le { 0.7483053690}.
	\]

	Now suppose $\frac1{\b_n}\le m+\frac12\le c{\b_n^{-\frac32}}$. Again pick $x_0=\frac{1}{750}$. This time we use $L=5$, $K=4$, $R=4$, and $\La=2$. {Recalling the definition of $f_j(u)$ in \eqref{fjdefn}, we} compute that $f_2(u)\ge0.2469769644$ for $u\ge1$. It follows that for $m+\frac12\ge\b_n^{-1}$,
	\[
		\mc M_{1,0,0,0}^{[5]} \ge 0.02167141828 \cdot \beta_n^4 e^{2\mathcal{C}_n-\left(m+\frac 12\right)\beta_n}.
	\]
	 Using the constant $C_{3c}$, we could bound the size of other terms $\Mc_{\ell,k,r,l}^{[5]}$ and treat them as error. We compute for $\b_n\le x_0$ that the contribution from $\Mc_{\ell,k,r,l}^{[5]}$ and all the error terms is bounded by
	\begin{multline*}
		\Dc_b(c) := \left({ 0.0004276419495}+0.004001738087c+0.01799670557c^2+0.1370655930c^3\right.\\
		\left.+0.8767499596c^4\right)\b_n^4e^{2\Cc_n-\left(m+\frac12\right)\b_n}.
	\end{multline*}
	Picking $c = \frac 13$, we compute that for $\beta_n \le x_0$, and $\frac1{\beta_n} \le m+\frac 12 \le {3^{-1}\beta_n^{-\frac 32}}$,
	\[
		\frac{\Dc_b\left(\frac13\right)}{\Mc_{1,0,0,0}^{[5]}} \le { 0.9072671370}.
	\]

	We conclude that for $\beta_n \le \frac{1}{750}$ and $m+\frac 12\le {3^{-1}\beta_n^{-\frac 32}}$, \eqref{rankdiff} holds. Using \eqref{eq:bL_def}, the condition $\beta_n \le \frac{1}{750}$ translates into $n\ge 925276$. Hence the theorem is established.
\end{proof}

\section{Partition inequalities and proof of \Cref{thm:main}}\label{FinalSection}

{In this section, we prove our main theorem, and thus establish Stanton's Conjecture. In the previous section, we prove positivity of rank differences for small values of $m$.
The following theorem treats unimodality of the rank function  if $m$ is large.

\begin{theorem}\label{thm:m_large}
Let $n\ge 11523$, and $2\sqrt{n}\le m \le n-3$. Then we have 
\begin{equation}\label{RankUniIneq}
N(m,n)-N(m+1,n) \ge 0.
\end{equation}
\end{theorem}

\begin{proof}
	Suppose that $n\ge11523$ and $m\ge2\sqrt{n}$. Define
	\[
		\om_r = \om_r(m,n) := n - \tfrac{3r^2}{2} - \left(m+\tfrac32\right)r.
	\]
	Using \eqref{eq:diff_gf} and the generating function for partitions it follows immediately that
	\begin{equation} \label{eq:Ndiff_partition}
		N(m,n) - N(m+1,n) = \sum_{r=1}^\infty (-1)^{r+1} (p(\omega_r)-2p(\omega_r+r)+p(\omega_r+2r)) = \sum_{r=1}^\infty (-1)^{r+1}P(\omega_r,r),
	\end{equation}
	where $P(\om_r,r)$ is given in \eqref{eq:wp_def}. Note that the sum above is actually finite.
	
	Now we split the theorem into a number of cases, which we consider individually.
	\begin{enumerate}[label={(\arabic*)},wide,labelwidth=!,labelindent=0pt]
		\item Suppose that $n-102 \le m \le n-3$. Then we have
		\[
			N(m,n) - N(m+1,n) = P(\om_1,1).
		\]
		We have $\om_1\le99$ and we check directly that $P(k,1)\ge0$ for $0\le k\le99$.
		
		\item Suppose that $\frac n2-2 \le m \le n-103$. Again we have
		\[
			N(m,n) - N(m+1,n) = P(\om_1,1).
		\]
		In this case we have $\om_1\ge100$, so \Cref{lem:sandwich} applies and we have
		\[
			P(\om_1,1) \ge \frac{e^{\pi\sqrt{\frac23\left(\om_1-\frac{1}{24}\right)}}}{6\left(\om_1+2-\frac{1}{24}\right)^2} > 0.
		\]
		
		\item Suppose that $\frac n2-54 \le m \le \frac n2-\frac 52$. Then we have
		\[
			N(m,n) - N(m+1,n) = P(\om_1,1) - P(\om_2,2).
		\]
		We have $\om_1\ge100$ and $\om_2\le99$ for $n\ge11523$. Applying \Cref{lem:sandwich} to $P(\om_1,1)$, we obtain
		\[
			P(\om_1,1) \ge \frac{e^{\pi\sqrt{\frac23\left(\om_1-\frac{1}{24}\right)}}}{6\left(\om_1+2-\frac{1}{24}\right)^2} \ge \frac{e^{\pi\sqrt{\frac23\left(\frac n2-\frac12-\frac{1}{24}\right)}}}{6\left(\frac n2+53-\frac{1}{24}\right)^2}.
		\]
		Meanwhile, we directly compute that $|P(\om_2,2)|\le11516573$ for $\om_2\le99$. Using these bounds, {we check the claimed inequality  \eqref{RankUniIneq}} {for $n\ge11523$ by a computer verification.} 
		
		\item Suppose that $\frac n3-\frac{11}{3} \le m \le \frac n2 - \frac{109}2$. Again we have 
		\[
			N(m,n) - N(m+1,n) = P(\om_1,1) - P(\om_2,2).
		\]
		In this case, we have $\om_1,\om_2\ge100$ so \Cref{lem:sandwich} applies for both $P(\om_1,1)$ and $P(\om_2,2)$. We obtain 
		\begin{align*}
			P(\om_1,1) &\ge \frac{e^{\pi\sqrt{\frac23\left(\om_1-\frac{1}{24}\right)}}}{6\left(\om_1+2-\frac{1}{24}\right)^2} \ge \frac{e^{\pi\sqrt{\frac23\left(\frac n2+\frac{103}{2}-\frac{1}{24}\right)}}}{6\left(\frac{2n}{3}+\frac83-\frac{1}{24}\right)^2},\\
			P(\om_2,2) &\le \frac{\pi^2e^{\pi\sqrt{\frac23\left(\om_2+4-\frac{1}{24}\right)}}}{6\sqrt{3}\left(\om_2-\frac{1}{24}\right)^2} \le \frac{\pi^2e^{\pi\sqrt{\frac23\left(\frac n3+\frac73-\frac{1}{24}\right)}}}{6\sqrt{3}\left(100-\frac{1}{24}\right)^2}.
		\end{align*}
		Using these bounds, {we numerically check \eqref{RankUniIneq} for $n\ge11523$.}
	\end{enumerate}
	
	We repeat these computations with $m$ in the ranges given below. The strategy for   $\om_r\ge100$ is as follows. We recall that we aim to give a lower bound on \eqref{eq:Ndiff_partition}, an alternating sum of the numbers $P(\om_r,r)$. Due to the presence of these alternating signs, we either require an upper or lower bound on $P(\om_r,r)$, which depends on the parity of $r$. \Cref{lem:sandwich} gives both an upper and lower bound on $P(\om_r,r)$, and for each value of $r$ we use the appropriate one. That is, we use the lower bound for odd values of $r$ and we use the upper bound there for the even values of $r$. For the smaller values $\om_r\le99$, we compute $\max_{-2r\le k\le99}|P(\om_r,r)|$, and treat them as error terms (regardless of the parity of $r$). Under the assumption that $n\ge11523$, we check by computer that the following ranges are all nonempty, and that there is at most one term for which \Cref{lem:sandwich} does not apply. Then {we check \eqref{RankUniIneq} by computer} for $n\ge11523$ for these cases. 
	\begin{enumerate}[label={(\arabic*)},wide,labelwidth=!,labelindent=0pt]
		\item If $\frac n3-39\le m\le\frac n3-4$, then \eqref{eq:Ndiff_partition} has $3$ terms, and \Cref{lem:sandwich} is not applicable to $P(\om_3,3)$.
		
		\item If $\frac n4-\frac{21}{4}\le m\le\frac n3-\frac{118}{3}$, then \eqref{eq:Ndiff_partition} has $3$ terms, and \Cref{lem:sandwich} is always applicable.
		
		\item If $\frac n4-\frac{129}{4}\le m\le\frac n4-\frac{11}{2}$, then \eqref{eq:Ndiff_partition} has $4$ terms, and \Cref{lem:sandwich} is not applicable to $P(\om_4,4)$.
		
		\item If $\frac n5-\frac{34}{5}\le m\le\frac n4-\frac{65}{2}$, then \eqref{eq:Ndiff_partition} has $4$ terms, and \Cref{lem:sandwich} is always applicable.
		
		\item If $\frac n5-\frac{144}{5}\le m\le\frac n5-7$, then \eqref{eq:Ndiff_partition} has $5$ terms, and \Cref{lem:sandwich} is not applicable to $P(\om_5,5)$.
		
		\item If $\frac n6-\frac{25}{3}\le m\le\frac n5-29$, then \eqref{eq:Ndiff_partition} has $5$ terms, and \Cref{lem:sandwich} is always applicable.
		
		\item If $\frac n6-27\le m\le\frac n6-\frac{17}{2}$, then \eqref{eq:Ndiff_partition} has $6$ terms, and \Cref{lem:sandwich} is not applicable to $P(\om_6,6)$.
		
		\item If $\frac n7-\frac{69}{7}\le m\le\frac n6-\frac{163}{6}$, then \eqref{eq:Ndiff_partition} has $6$ terms, and \Cref{lem:sandwich} is always applicable.
		
		\item If $\frac n7-\frac{183}{7}\le m\le\frac n7-10$, then \eqref{eq:Ndiff_partition} has $7$ terms, and \Cref{lem:sandwich} is not applicable to $P(\om_7,7)$.
		
		\item If $\frac n8-\frac{91}{8}\le m\le\frac n7-\frac{184}{7}$, then \eqref{eq:Ndiff_partition} has $7$ terms, and \Cref{lem:sandwich} is always applicable.
	\end{enumerate}
	For all the cases above, {we verify \eqref{RankUniIneq} by computer} if the said conditions are satisfied.

	It remains to consider the case where $2\sqrt{n} \le m\le\frac n8-\frac{23}{2}$. We split the sum \eqref{eq:Ndiff_partition} into two parts:
	\begin{equation}\label{eq:partition_split}
		\sum_{r=1}^\infty (-1)^{r+1}P(\om_r,r) = \sum_{\substack{r\ge1\\\om_r\ge\frac{3n}{4} + \frac 1{24}}} (-1)^{r+1}P(\om_r,r) + \sum_{\substack{r\ge1\\\om_r<\frac{3n}{4} + \frac 1{24}}} (-1)^{r+1}P(\om_r,r).
	\end{equation}
	Consider the terms (and discarding at most one positive term), we have that
	\begin{align}\label{eq:partition_paired}
		\sum_{\substack{r\ge1\\\om_r\ge\frac{3n}{4}+\frac{1}{24}}} (-1)^{r+1}P(\om_r,r) &\ge \sum_{\substack{r\ge1\\\om_{2r}\ge\frac{3n}{4}+\frac{1}{24}}} (P(\om_{2r-1},2r-1)-P(\om_{2r},2r))\\
		\nonumber
		&\ge \sum_{\substack{r\ge1\\\om_{2r}\ge\frac{3n}{4}+\frac{1}{24}}} \left(\frac{(2r-1)^2e^{\pi\sqrt{\frac23\left(\om_{2r-1}-\frac{1}{24}\right)}}} {6\left(\om_{2r-1}+4r-2-\frac{1}{24}\right)^2} - \frac{\pi^2(2r)^2e^{\pi\sqrt{\frac23\left(\om_{2r}+4r-\frac{1}{24}\right)}}} {24\sqrt{3}\left(\om_{2r}-\frac{1}{24}\right)^2}\right),
	\end{align}
	applying \Cref{lem:sandwich} for the last inequality. Note that $\om_2=n-2m-9\ge\frac{3n}{4}+4\ge\frac{3n}{4}+\frac{1}{24}$, so this sum is always nonempty. In order to have each summand being positive, we need 
	\begin{equation}\label{eq:partition_comparer}
		e^{\pi\sqrt{\frac23}\rb{\sqrt{\om_{2r-1}-\frac{1}{24}}-\sqrt{\om_{2r}+4r-\frac{1}{24}}}} \ge \frac{\pi^2}{4\sqrt{3}} \frac{(2r)^2}{(2r-1)^2} \frac{\left(\om_{2r-1}+2(2r-1)-\frac{1}{24}\right)^2}{\left(\om_{2r}-\frac{1}{24}\right)^2}.
	\end{equation}
	To see this, note that
	\begin{equation}\label{eq:169}
		\frac{\left(\om_{2r-1}+2(2r-1)-\frac{1}{24}\right)^2}{\left(\om_{2r}-\frac{1}{24}\right)^2} < \frac{16}{9}.
	\end{equation}
	{Since $\frac{(2r)^2}{(2r-1)^2} \le 4$ for $r\ge 1$, we deduce that the right-hand side of \eqref{eq:partition_comparer} is bounded from above by $\frac{16\pi^2}{9\sqrt{3}}$.}
	On the other hand, using the definition of $\om_r$, we compute
	\begin{equation}\label{eq:rmr}
		\sqrt{\omega_{2r-1}-\frac 1{24}} - \sqrt{\omega_{2r}+4r-\frac 1{24}} \ge \frac{2r+m}{2\sqrt{n}}.
	\end{equation}
	Hence, if $m\ge 2\sqrt{n}$, then the left-hand side of \eqref{eq:partition_comparer} is bounded from below by {$e^{\pi\sqrt{\frac23}}$. Meanwhile, by direct computation we have that $e^{\pi\sqrt{\frac23}} > \frac{16\pi^2}{9\sqrt{3}}$. So the inequality \eqref{eq:partition_comparer} holds.} 

	Now we look at the second sum in \eqref{eq:partition_split}. In the following we show that this sum is bounded by the term $r=1$ on the right-hand side in \eqref{eq:partition_paired}. This then implies the non-negativity of \eqref{eq:partition_split}, and proves the theorem in this case. We first make some observations.
	\begin{enumerate}[label={(\arabic*)},wide, labelwidth=!,labelindent=0pt]
		\item If $n\ge11523$, $m\ge2\sqrt{n}$ and $\om_r+2r>0$, then we have $r\le\sqrt{\frac{2n}{3}}$.
		
		\item The partition function $p(n)$ is monotonically increasing.
	\end{enumerate}
	Using the observations above, we obtain a trivial bound
	\begin{align}
		\sum_{\substack{r\ge 1\\ \omega_r < \frac {3n}4+\frac 1{24}}} \vb{P(\omega_r,r)} &\le \sum_{\substack{r\ge 1\\ \omega_r < \frac {3n}4+\frac 1{24}}} \rb{p(\omega_r) + {2}p(\omega_r+r) + p(\omega_r+2r)}\nonumber\\
		&\le \sqrt{\frac{2n}{3}} \cdot 4 \cdot p\rb{\left\lfloor\tfrac{3n}{4}+2\sqrt{\tfrac{2n}{3}} + \tfrac 1{24} \right\rfloor} \le \frac{\sqrt{2n}\cdot e^{\pi\sqrt{\frac23\left(\frac{3n}{4}+2\sqrt{\frac{2n}{3}}\right)}}}{3\left(\frac{3n}{4} + 2\sqrt{\frac{2n}{3}}\right)},\label{eq:Psum_prebound}
	\end{align}
	the final inequality following from \eqref{eq:partition_ubound}.
	
	Now we show that this error is bounded by the term $r=1$ on the right-hand side in \eqref{eq:partition_paired}. Using \eqref{eq:169} and \eqref{eq:rmr}, the first term on the right-hand side of \eqref{eq:partition_paired} is bounded below by
	\[
		\frac{e^{\pi\sqrt{\frac23\left(\om_1-\frac{1}{24}\right)}}}{6\left(\om_1+2- \frac{1}{24}\right)^2} \left(1-\frac{16\pi^2}{9\sqrt{3}}e^{-\pi\sqrt{\frac23}}\right).
	\]
	Now recall the assumption $2\sqrt{n}\le m\le\frac n8-\frac{23}{2}$. By minimizing the exponent and maximizing the denominator, we obtain a lower bound
	\[
		\frac{e^{\pi\sqrt{\frac23\left(\om_1-\frac{1}{24}\right)}}}{6\left(\om_1+2- \frac{1}{24}\right)^2} \left(1-\frac{16\pi^2}{9\sqrt{3}}e^{-\pi\sqrt{\frac23}}\right) \ge \frac{e^{\pi\sqrt{\frac23\left(\frac{7n}{8}+\frac{17}{2}-\frac{1}{24}\right)}}} {6\left(n-2\sqrt{n}-1-\frac{1}{24}\right)^2} \left(1-\frac{16\pi^2}{9\sqrt{3}}e^{-\pi\sqrt{\frac23}}\right).
	\]
	By direct computation, we check that for $n\ge 11523$,
	\[
		\frac{\sqrt{2n}\cdot e^{\pi\sqrt{\frac{{\frac{3n}{4}+2\sqrt{\frac{2n}{3}}}}{3}}}}{3\left(\frac{3n}{4}+ 2\sqrt{\frac{2n}{3}}\right)} \le \frac{e^{\pi\sqrt{\frac23\left(\frac{7n}{8}+\frac{17}{2}- \frac{1}{24}\right)}}}{6\left(n-2\sqrt{n}-1-\frac{1}{24}\right)^2} \left(1-\frac{16\pi^2}{9\sqrt{3}}e^{-\pi\sqrt{\frac23}}\right).
	\]
	 Hence the second sum in \eqref{eq:partition_split} is bounded by the first sum in \eqref{eq:partition_split}, which necessarily exists (see the line after \eqref{eq:partition_paired}) and is positive. This yields the claim.
\end{proof}

We now combine the small and large $m$ bounds to prove \Cref{thm:main}.

\begin{proof}[Proof of \Cref{thm:main}]
	 { By \Cref{thm:m_small} the claim holds for $n\ge 925276$ and} { $0\le m \le  {3^{-1}\beta_n^{-\frac 32}} - \frac 12$. On the other hand, by \Cref{thm:m_large} for $n\ge 11523$ and $2\sqrt{n}\le m \le n-3$. Observing that $2\sqrt{n}\le{3^{-1}\b_n^{-\frac32}}-\frac12$ for $n\ge925276$, we see that it remains to check the following cases:}
	\begin{enumerate}[label={(\arabic*)},wide, labelwidth=!,labelindent=0pt]
		\item $n\le 11522$, $0\le m \le n-3$,
		
		\item $11523 \le n \le 925275$, $0\le m < 2\sqrt{n}$.
	\end{enumerate}
{This is a large, but finite number of cases to check. This can be performed as follows.} An efficient way to do this is via the identity \eqref{eq:Ndiff_partition}, which computes the difference $N(m,n)-N(m+1,n)$ in $O(n^\frac12)$ operations from the partition function $p(n)$, which is widely accessible as a list and easily computable. {Performing this computer check completes the proof of the theorem.}
\end{proof}


\providecommand{\bysame}{\leavevmode\hbox to3em{\hrulefill}\thinspace}
\providecommand{\MR}{\relax\ifhmode\unskip\space\fi MR }
\providecommand{\MRhref}[2]{%
  \href{http://www.ams.org/mathscinet-getitem?mr=#1}{#2}
}
\providecommand{\href}[2]{#2}

\end{document}